\documentclass{amsart}

\usepackage{amssymb,dirtytalk,comment,tikz-cd,color}

\usepackage[T1]{fontenc}
\usepackage[italian,british]{babel}

\usepackage[hidelinks]{hyperref}

\newtheorem*{theorem*}{Theorem}
\newtheorem{theorem}{Theorem}[section]
\newtheorem{lemma}[theorem]{Lemma}
\newtheorem{prop}[theorem]{Proposition}
\newtheorem{corollary}[theorem]{Corollary}

\theoremstyle{definition}
\newtheorem{definition}[theorem]{Definition}
\newtheorem{example}[theorem]{Example}

\newtheorem*{problem*}{Problem}

\theoremstyle{remark}
\newtheorem{remark}[theorem]{Remark}
\newtheorem{notation}[theorem]{Notation}

\numberwithin{equation}{section}

\DeclareMathOperator{\Ann}{Ann}

\DeclareMathOperator{\tors}{tors}
\DeclareMathOperator{\TN}{TN}

\DeclareMathOperator{\rank}{rank}
\DeclareMathOperator{\Prank}{P-rank}

\newcommand{\ator}{{(A^*)}_{\tors}}

\newcommand{\n}{\mathfrak {N}}
\newcommand{\nt}{\mathfrak {N}_{\tors}}
\newcommand{\ntp}{\mathfrak {N}_{\tors,p}}

\newcommand{\N}{\mathbb{N}}
\newcommand{\Z}{\mathbb{Z}}
\newcommand{\Q}{\mathbb{Q}}
\newcommand{\F}{\mathbb{F}}

\newcommand{\fm}{\mathfrak{m}}

\newcommand{\geps}{\mathcal G(\varepsilon)}

\begin{document}

\title[Fuchs' problem in the small torsion case]
      {On Fuchs' problem for finitely generated abelian groups: The small torsion case}

\author{I.~Del Corso}
\address{Dipartimento di Matematica,
          Universit\`a di Pisa,
          Largo Bruno Pontecorvo 5, 56127 Pisa, Italy}
\email{ilaria.delcorso@unipi.it}
\urladdr{http://people.dm.unipi.it/delcorso/}

\author{L.~Stefanello}
\address{Dipartimento di Matematica,
          Universit\`a di Pisa,
          Largo Bruno Pontecorvo 5, 56127 Pisa, Italy}
\email{lorenzo.stefanello@phd.unipi.it}
\urladdr{https://people.dm.unipi.it/stefanello/}
\thanks{Both authors were a member of GNSAGA - INdAM and  gratefully acknowledge support from the Departments of Mathematics of the Universities of Pisa, and  the MIUR Excellence Department Project awarded to the Department of Mathematics, University of Pisa, CUP I57G22000700001.
The authors have performed this activity in the framework of the PRIN 2022, title \say{Semiabelian varieties, Galois representations and related Diophantine problems}.}



\begin{abstract}
A classical problem, raised by Fuchs in 1960, asks to classify the abelian groups which are  groups of units of some rings. 

In this paper, we consider the case of finitely generated abelian groups, solving Fuchs' problem for such group with the additional  assumption that the torsion subgroups are  \emph{small}, for a suitable notion of small related to the Pr\"ufer rank. As a concrete instance,  we classify for each $n\ge2$ the realisable groups of the form $\Z/n\Z\times\Z^r$.

Our tools require an investigation of the adjoint group of suitable radical rings of odd prime power order appearing in the picture, giving conditions under which the additive and adjoint groups are isomorphic. 

In the last section, we also deal with some groups of order a power of $2$, proving that the groups of the form $\Z/4\Z\times \Z/2^{u}\Z$ are realisable if and only if  $0\le u\le 3$ or $2^u+1$ is a Fermat's prime.
\end{abstract}

\maketitle

\section{Introduction}
\label{section: introduction}

\subsection{The study of the group of units in the literature} The study of the group of units of unitary rings is a classical question. A milestone in this context is  the celebrated Dirichlet's Unit Theorem (1846), which shows that the group of units of the ring of integers of a number field $K$ is a finitely generated abelian group of the form $\Z/{2n}\Z\times \Z^g$, where $n\ge 1$ and $g$ are determined by the structure of the field. 

In a perfect analogue of this, Higman proved in 1940 that, given a finite abelian group $T$, the group ring $\Z T$ has group of units isomorphic to $\pm T\times \Z^g$ for a suitable explicit constant $g$.

The problem of the characterisation of the groups of unity was posed explicitly by Fuchs in~\cite[Problem 72]{Fuchs60}.

\begin{problem*}[Fuchs' problem]
	Characterise the groups which are the group of all units in a commutative and associative ring with identity. 
\end{problem*}

This problem has been considered by various authors in subsequent years. We mention here some contributions. 

In~\cite{Gilmer63}, Gilmer considered the case of finite commutative rings, classifying the possible cyclic groups that arise in this case. Later, Fuchs' problem was addressed by Hallett and Hirsch~\cite{HallettHirsch65}, and subsequently by Hirsch and Zassenhaus~\cite{HirschZassenhaus66}; a combination of their results  with~\cite{Corner63} implies that if a finite group is the group of units of a reduced torsion-free ring, then it must satisfy some necessary conditions, namely, it must be a subgroup of a direct product of groups of a given family.

Fuchs' problem was then entirely solved for finite cyclic groups by Pearson and Schneider~\cite{PearsonSchneider70}, who combined the result of Gilmer and the result of Hallett and Hirsch, and recently for indecomposable abelian groups by Chebolu and Lockridge~\cite{CheboluLockridge15}.

Pearson and Schneider's approach also shows that  for the addressing  of Fuchs' problem in its generality one can reduce to consider finite rings  and a special class of characteristic zero rings, which we term TN rings here. Motivated by this, in the papers~\cite{dcdAMPA,dcdBLMS} Dvornicich and the first author studied Fuchs' problem for rings of positive characteristic and of characteristic zero respectively,  obtaining necessary conditions for a finite abelian group to be realisable as the group of all units of a ring (realisable, for short), and producing infinite families of both realisable and non-realisable groups. Later, the first author considered in~\cite{JLMS} Fuchs' problem for finitely generated abelian groups, completely classifying the finitely generated abelian groups realisable as group of all units for the classes of integral domains, torsion free rings, and reduced rings. The main result of that paper is the classification of the finitely generated abelian groups that  can be realised in the class of torsion free rings, which turn out to be all groups of the form  
$$T\times\Z^r,$$
where $T$ is a finite abelian group of even order, and $r$ any value greater than or equal to $g(T)$, an explicit constant whose value depends on the structure of the group $T$ (see \cite[Theorem 5.1]{JLMS}, or Theorem~\ref{teo:thm5.1JLMS} in this paper).

\subsection{The present contribution} In this paper, we approach Fuchs' problem for finitely generated abelian groups without any restriction on the class of rings. The methods developed in the paper allow us   to  classify the realisable finitely generated abelian groups with \emph{small} torsion part, for a suitable notion of small related to the Pr\"ufer rank.

Our investigation proceeds via the construction and the study of an exact sequence, in the two basic cases of the finite local rings and the TN rings, from which the general results can be deduced. In fact, if an abelian group is realisable as the group of units of some ring $A$, then it fits as the middle group in an exact sequence of the form 
\begin{equation}
\label{eq:intro}
1\to 1+\mathfrak{I}\hookrightarrow A^*{\to}\left(A/\mathfrak{I}\right)^*\to 1,
\end{equation}
where $\mathfrak I$ is an ideal of $A$ contained into its nilradical $\n$ (see Proposition~\ref{prop:successioneesatta}). 
The idea is that, for suitable choices for $\mathfrak I$,  the units of  $A/\mathfrak{I}$ can be described by appealing to the results \cite{dcdAMPA} and \cite{JLMS}.

The other \say{ingredient} of the exact sequence is  $1+\mathfrak{I}$, which is the adjoint group of the radical ring $\mathfrak I$. We have some  information about the group structure of $\mathfrak I$, and it is known that it must influence  the structure of its adjoint group $1+\mathfrak I$, somehow. However, the exact relationship is mostly unknown, and this remains the challenging aspect of the sequence to describe, in general.
 The methods we develop in Section~\ref{section: radical}  allows us to  
prove that if one of the two groups is \emph{small}, then  $\mathfrak I$ and $1+\mathfrak I$ are isomorphic (see Theorem~\ref{theorem: main}). In this paper we build upon this result.

When $A$ is a finite local ring, then from \cite[Theorem 3.1]{dcdAMPA} we know that, by choosing  $\mathfrak I$ to be the maximal ideal $\mathfrak m$, the sequence splits. Therefore the realisable groups are of those of the form 
$\F_{p^\lambda}^*\times1+\mathfrak m$, where $p$ is a prime and $\lambda\ge1$, so the classification depends on the possible structure of the $p$-group $1+\mathfrak m$: the results on radical ring we discussed above allows us to characterise this group in the case when it is small (see Theorem~\ref{theorem: main finite case}).

The class of TN rings  contains the class of torsion-free rings, and since we know that all groups of the form $T\times\Z^r$, where $T$ is finite of even order, are realisable by torsion-free rings when the rank $r$ is big enough,  the question is to determine, for each $T$, the minimum value $r(T)$ for the rank in this wider class.
In this case, we find it useful to specialise the exact sequence by choosing  $\mathfrak I=\nt$, the torsion subgroup of the nilradical, so that $A/\nt$ is a torsion-free ring, whose group of units  has the same rank of $A^*$.  This choice turns out to be a good one also for the study   $1+\nt$, which in the relevant cases is finite, and our methods of Section~\ref{section: radical} can be applied, when it is small. The exact sequence  \eqref{eq:intro} is used in the paper both for bounding the rank from below for a fixed torsion part $T,$ and for the construction of examples realising the minimum value $r(T)$ of the rank.

Our main result for TN rings is Theorem~\ref{theorem: main realisable} in which we prove that 
a group of the form $T\times\Z^r$,  where $T$ is a finite abelian group of even order with small Sylow subgroups, is realisable by a TN ring if and only if $r\ge r(T)$, where $r(T)$ is an explicit constant whose value depends on the structure of the group $T$.

As a corollary we get a complete classification of the realisable groups of the form $\Z/n\Z\times\Z^r$ (see Theorem~\ref{teo:zn} and Corollary~\ref{corollary: all the cyclic}). This result extends Pearson and Schneider's results on cyclic groups.

The paper is organised as follows. In Section~\ref{sec:prel} we introduce the notation and  some preliminary results; we propose in particular some reduction steps to deal with Fuchs' problem, in which the role of the adjoint groups of radical rings starts to emerge. In Section~\ref{section: radical} we deal with the study of radical rings; we show in particular that in the small case, the additive and the adjoint group of a radical ring needs to be isomorphic. We employ this observation to deal with finite rings in Section~\ref{sec:finiti} and with TN rings in Section~\ref{section: tn units}, thanks also to some tools on torsion-free rings that we recall and specialise to our aim,  in Section~\ref{section: tn tools}. Finally, Section~\ref{sec:2-groups} is devoted to some $2$-groups; specifically, we solve Fuchs' problem for groups of the form $\Z/4\Z\times \Z/2^{u}\Z$, for $u\ge 0$. This is a challenging case for Fuchs' problem, as the result we develop in the rest of the work, related to the short exact sequence and the small rank, leave untouched the case $p=2$.

\section{Notation and preliminary results}
\label{sec:prel}

Throughout the paper we shall adopt  the following classical  notation: for $q$ a power of a prime, $\F_q$ denote a field with $q$ elements; $\Z/n\Z$ ($n\ge2$) is the ring of classes modulo $n$; and by $\zeta_n$ we denote any complex primitive $n$th root  of $1$.

All rings considered in this work (except when radical rings) are assumed to be associative and unitary, and to have finitely generated abelian group of units.

Given a ring $A$, we denote by $A^*$ the group of units, and by  $A_0$ its fundamental  subring, namely, $A_0=\Z$ or $\Z/n\Z$ depending on whether the characteristic of $A$ is $0$ or $n>0$.

In the context of Fuchs' problem, a group $G$ is called \emph{realisable} (in a certain class $\mathcal C$ of rings) if there exists a ring $A$ (in $\mathcal C$) such that $A^*$ is isomorphic to $G$. 

In this paper we are interested in studying Fuchs' problem for finitely generated abelian groups, namely groups $G$ that, up to isomorphism, can be written as
$$T\times \Z^r$$
where $T$ is a finite abelian group and $r\ge 0$ is the rank of $G$, written $r=\rank G$. Such a group is realisable if there exists a ring $A$ such that $(A^*)_{\tors}$, the torsion subgroup of  $A^*$, is isomorphic to $T$,  and  the rank of $A^*$ is $r$.

\subsection{A reduction step}
The following proposition shows that all the realisable finitely generated abelian groups can already be obtained  in a more manageable  subclass of  rings. 
\begin{prop}
\label{prop:ABC}
Let $A$ be a ring, and let $B=A_0[A^*]$ and $C=A_0[\ator]$ be the subrings of $A$ generated over $A_0$ by the units and the torsion units of $A$, respectively.
Then $B^*=A^*$ and $C^*\cong\ator\times\Z^{r_C}$ with $r_C\le \rank A^*$.
\end{prop}
\begin{proof}
We have $B\subseteq A$ and $C\subseteq A$, therefore $B^*\le A^*$ and  $C^*\le A^*$. 
Since,  by construction, $A^*\subseteq B^*$, we get $B^*=A^*$.

On the other hand, $\ator\le C^*\le A^*$, so that $C^*$ is a finitely generated abelian group with $(C^*)_{\tors}=\ator$ and $r_C\le r_A$. 
\end{proof}
\begin{corollary}
\label{cor:sample-rings}
Let $T$ be a finite abelian group and $r$ a non negative integer. The following facts are equivalent.
\begin{enumerate}
\item The group $T\times\Z^r$ is realisable.
\item The group $T\times\Z^r$ is realisable in the class of commutative rings which are finitely generated over their fundamental subring.
\item There exists a commutative ring finitely generated and integral over its fundamental subring with group of units isomorphic to  $T\times\Z^s$ for some $s\le r$.
\end{enumerate}
\end{corollary}
\begin{proof}
We keep  the notation of Proposition \ref{prop:ABC}. The ring   $B=A_0[A^*]$ is commutative finitely generated over $A_0$ and has the same group of units as $A$, proving that $(1)\iff(2)$.

The ring $C=A_0[\ator]$ is finitely generated and integral over $A_0$: in fact, every torsion unit is an element of finite multiplicative order, so it is a root of a polynomial of type $x^n-1$, therefore it is  integral over $A_0$, so  the previous proposition shows that $(1)\Rightarrow(3)$. As for the converse, we appeal to \cite[Theorem 3.3]{CheboluLockridge15} which guarantees that $\F_2[\Z]\cong\Z$.
\end{proof}

\begin{remark} 
The last corollary shows that to classify the finitely generated abelian groups that are realisable, we can limit ourself to the study of  the class of commutative rings that are finitely generated and integral over their fundamental subring $A_0$. (Moreover, the generators can also be chosen to be elements of finite multiplicative order.)  This plays a key role in a second reduction step; see Proposition~\ref{prop:ps} below.
 \end{remark}

\begin{corollary}
\label{cor:finiti}
A finitely generated abelian group is realisable in the class of rings of positive characteristic if and only if its torsion subgroup is realisable in the class of finite rings.
\end{corollary}
\begin{proof}
With the notation of Proposition \ref{prop:ABC}, if $A_0$ is finite, then $C=A_0[\ator]$ is finite, so $C^*=\ator$.
\end{proof}

\begin{remark} 
\label{rem:realizzabili}
For the realisability of groups of type $T\times\Z^r$, where $T$ is a finite abelian group and $r\ge0$,  there is a completely different behaviour between  the classes of
 characteristic zero and positive characteristic rings. In fact, by Corollary~\ref{cor:finiti}, $T\times\Z^r$ is realisable in positive characteristic only if $T$ itself is  realisable as  the group of units of a (finite) ring, and all the results of \cite{dcdAMPA} and of Section~\ref{sec:finiti} apply in this case. In particular,  not all finite abelian groups $T$ can occur.

Instead, when considering rings of characteristic zero, in order for the group  $T\times \Z^r$ to be realisable for some $r$, the group $T$ must have even order (since $-1$ is a unit of order 2), and all the $T$'s of even order appear as the torsion part of a finitely generated abelian group of units (see \cite[Theorem~5.1]{JLMS}).
However, not all finite groups  are realisable as a group of units, namely  if we  require that $r=0$  (see  \cite{dcdBLMS}).
Therefore, for characteristic zero rings  the question is to determine the possible $r$  such that $T\times \Z^{r}$ is the group of units of some ring $A$, and Corollary~\ref{cor:sample-rings} says that to  answer this question it is sufficient to consider the  finitely generated integral extensions of $\Z$.
\end{remark}

\subsection{A further reduction step}
In view of our previous reduction, we can limit our study to the class of commutative rings that are finitely generated and integral over their fundamental subring. The following Proposition \ref{prop:ps}, due to Pearson and Schneider, shows that each ring in this class splits into a direct product of a finite ring and a rings  whose torsion elements are nilpotent. (We recall that an element $a\in A$ is a torsion element if its additive order is finite, and a commutative ring $A$ is \emph{torsion-free} if $0$ is its only torsion element. An element $a\in A$ is nilpotent if there exists $n$ such that $a^n=0$; the set ideal of nilpotent elements of $A$ is called \emph{nilradical}, and $A$ is \emph{reduced} if the nilradical is trivial.)
\begin{prop}\cite[Proposition 1]{PearsonSchneider70}
\label{prop:ps} 
Let $A$ be a commutative ring which is finitely generated and integral over its fundamental subring. Then $A$ is isomorphic to a ring of the form $A_1\times A_2$, where $A_1$ is a finite ring and the torsion ideal of $A_2$ is contained in its nilradical.
\end{prop}

\begin{definition}
A commutative ring $A$ with $A^*$ finitely generated is called a \emph{TN ring} if its torsion ideal is contained in the nilradical.
\end{definition}
\begin{remark}
\label{rem:PS}
A TN  ring has characteristic $0$, since otherwise 1 would
    be a torsion element and hence nilpotent. Note also that if $A$ is a $\TN$ ring, then also $\Z[A^{*}_{\tors}]$ is a TN ring.
\end{remark}

Proposition \ref{prop:ps} permits to split the study of Fuchs' problem into the study of the realisability of  finite generated abelian  group in the classes of finite rings and TN rings. 
The units of both these type of rings have been widely studied. In particular, in the case of finite groups a systematic, even if non exhaustive, study  has been carried on in \cite{dcdAMPA} for finite characteristic rings, of which finite rings are a subclass, and in  \cite{dcdBLMS} for characteristic 0 rings, of which TN rings are a subclass. In addition, both classes were also considered in~\cite{JLMS}, where the groups studied are finitely generated and abelian.

These papers, of which we recall the main results in the next sections, are the starting point of the present study, which improves the previous results for both classes of rings.

\subsection{An exact sequence}
For a commutative ring $A$, we denote by $\mathfrak{N}=\mathfrak{N}(A)$ its nilradical. In the following we shall make repeated use of the following proposition (see \cite[Proposition 2.2]{dcdAMPA}).
\begin{prop}
\label{prop:successioneesatta}
{\sl
Let $A$ be a commutative ring, and let $\mathfrak{I}$ an ideal of $A$ contained in $\mathfrak{N}$. Then the sequence
\begin{equation}
\label{eq:success}
1\to 1+\mathfrak{I}\hookrightarrow A^*\stackrel{\phi}{\to}\left(A/\mathfrak{I}\right)^*\to 1,
\end{equation}
where $\phi(x)=x+\mathfrak{I}$, is exact.
}
\end{prop}

\begin{remark}
\label{rem:split}
The previous proposition suggests addressing the problem of studying the units of $A$ through the study of the units  of one of its quotient of the form $A/\mathfrak I$,  and the study of $1+\mathfrak I$, for some $\mathfrak I\subseteq \n$.

In view of the results of \cite{JLMS}, the study of $(A/\mathfrak I)^*$ can be done for some special choice for $\mathfrak I$: for example for $\mathfrak I= \n$ the ring $A/\n$ is a reduced ring and  in \cite[Theorem 6.4]{JLMS} there is the classification of  the possible groups of units that can be realised by these rings.

However, our choice of $\mathfrak I\subseteq \n$ is also conditioned by the fact that we need to study the group $1+\mathfrak I$. Now, since $\mathfrak I$ consists of nilpotent elements, $1+\mathfrak I$ is the adjoint  group of the radical ring $\mathfrak I$ (see the next section for the definition).
In the literature we can find some result for the study of $\mathfrak I$, and in the next section, we will develop some new ones, which will be useful in the present case.

In the end, our choice for $\mathfrak I$ will be the result of a balance that allows us to have information both on $(A/\mathfrak I)^*$ and on $1+\mathfrak I$.

As a final remark, we notice that in general the sequence \eqref{eq:success} does not split (see Example~\ref{example: strange behaviour}).
\end{remark}

\subsection{Some results in number theory}
We state here, mainly without proofs, some standard results in number theory we will need in the paper.
All of them can be found in any standard number theory book (see for example \cite{FT93}).
\subsubsection{p-adic rings}
Let $p$ be  a prime number. Denote by $\Q_p$ the field of $p$-adic number and by $\Z_p$ its ring of integers, namely the integral closure of $\Z$ into $\Q_p$.  For each $\lambda\ge 1$ the field $\Q_p$ admits a unique unramified extension $U_\lambda$ of degree $\lambda$ (inside a fixed algebraic closure): we denote   by $\Z_p(\lambda)$ its ring of integers, namely the integral closure of $\Z_p$ into $U_\lambda$. Clearly, $\Z_p(1)=\Z_p$. The ring $\Z_p(\lambda)$ has the following properties.
\begin{itemize}
	\item $\Z_p(\lambda)$ is a discrete valuation ring, and therefore a principal ideal domain (PID);
	\item $p\Z_p(\lambda)$ is the only non-zero prime ideal of $\Z_p(\lambda)$ and every non-zero ideal of  $\Z_p(\lambda)$ is of the form $p^a\Z_p(\lambda)$ for some $a>0$;

	\item for all $a\ge 0$, there exists a group isomorphism
\begin{equation*}
	\frac{\Z_p(\lambda)}{p^a\Z_p(\lambda)}\cong \left(\frac{\Z}{p^a\Z}\right)^{\lambda}
\end{equation*}
(this can be easily obtained by noticing that $\Z_p(\lambda)$ is a free $\Z_p$-module of rank $\lambda$).
\end{itemize}    
\begin{prop}\label{proposition: a Zp-module is lambda power}
	Every $\Z_p(\lambda)$-module is the $\lambda$-power of an abelian $p$-group.
\end{prop}
\begin{proof}
The structure theorem for finitely generated modules over a PID implies that $N$ is isomorphic to the direct sum of modules of the form
\begin{equation*}
	\frac{\Z_p(\lambda)}{p^{a}\Z_p(\lambda)},
\end{equation*}with $a>0$. As there is a group isomorphism
\begin{equation*}
	\frac{\Z_p(\lambda)}{p^{a}\Z_p(\lambda)}\cong \left(\frac{\Z}{p^{a}\Z}\right)^\lambda,
\end{equation*} we conclude that $N$ is the  is the  $\lambda$-power of an abelian $p$-group. 
\end{proof}

\subsubsection{Cyclotomic rings}
Let $K$ be a finite extension of $\Q$, and denote by $\mathcal O_K$ its ring of integers. 
The general theory guarantees that $\mathcal O_K$ is a Dedekind domain, so all its non-zero ideals factors into a product of prime ideals. Moreover,  and all its quotients $\mathcal O_K/I$, where $\{0\}\ne I\subset \mathcal O_K$ is an ideal, are finite and can be described in terms of the factorisation of $I$.

We are interested in the case where $K=\Q(\zeta_n)$; for such $K$ it is known that $\mathcal O_K=\Z[\zeta_n]$. The next proposition specifies the previous information for the quotients   of $\Z[\zeta_n]$ (by  ``ideal'' we  mean non-zero ideal).

\begin{notation}
	Given a prime $p$ coprime to $n$, we write $\lambda(p,n)$ for the order of $p$ modulo $n$.
\end{notation}
   		
  \begin{prop}
  \label{proposition: cyclotomic quotients}
 	Let $n\ge 1$. 
 	\begin{itemize}
	\item Let $I$ be an ideal of $\Z[\zeta_n]$, and let $I=\prod_{i=1}^sP_i^{a_i}$ be its factorisation into pairwise distinct prime ideals. Then there is a ring isomorphism
 		\begin{equation*}
 			\frac{\Z[\zeta_n]}{I}\cong \prod_{i=1}^s \frac{\Z[\zeta_n]}{P_i^{a_i}}.
 		\end{equation*}
 		\item Let $P\subset\Z[\zeta_n]$ be a prime ideal, and let  $(p)=P\cap \Z$. Then for all $a>0$,  $\Z[\zeta_n]/P^a$ is a finite commutative local ring of order $p^{a\lambda(p,n)}$. In addition, if $p$ is coprime to $n$, then there is a group isomorphism
		\begin{equation*}
 			\frac{\Z[\zeta_n]}{P^a}\cong ({\Z}/{p^a\Z})^{\lambda(p,n)}. 
 		\end{equation*} 
 		\end{itemize}
 \end{prop}


\section{Radical rings}\label{section: radical}

Recall that a ring $N$ (necessarily without unity) is called \emph{radical} if $N$ equals its Jacobson radical. As shown in~\cite[Chapter IX, Theorem 2.3]{hungerford}, this is equivalent to asking that 
\begin{equation*}
	1+N=\{1+x\mid x\in N\}
\end{equation*}is a group under the multiplication of the ring, or that $(N,\circ)$ is a group, where
\begin{equation*}
	x\circ y=x+y+xy.
\end{equation*}
Note that $(N,\circ)\cong 1+N$, and we call this the \emph{adjoint group} of $N$.

For example, given a ring $A$, every ideal of $A$ consisting of nilpotent elements is a radical ring.

An important question in the study of radical ring and  the related problems is understanding the relation between the additive group of the ring and the adjoint group (see for example \cite{AD95,Dic97,FCC}). A possible approach consists of realising that this problem is a specific version of a more general one. Given a radical ring $N$, the left distributivity of the multiplication over the sum immediately yields 
\begin{equation*}
		x\circ(y+z)=(x\circ y)-x+(x\circ z) 
	\end{equation*}
	for all $x,y,z\in N$. This means that $(N,+,\circ)$ is a \emph{brace}; see~\cite{Rum07,CJO14}. The relation between the operations of a brace was the object of study in~\cite{Bac16, Byo15, Nas19, TQ20, CDC22}, for example. In the recent work~\cite{DC23}, this study was taken in consideration for a suitable class of braces, termed \emph{module braces}; for a radical ring $N$, this notion boils down to the usual notion of module over a (commutative) ring; see~\cite[Example 2]{DC23}.  More specifically, the main idea of~\cite{DC23} is to require some properties in relation to the rank of the groups, to find an interaction between them. Recall the following notion: given a PID $D$ and a finite $D$-module $M$ of order a power or a prime $p$, we can write $M$ as direct sum of indecomposable cyclic $D$-modules, and we say that the \emph{$D$-Pr{\"u}fer rank} $\Prank_D M$ of $M$ is the number of modules in this decomposition. (Note that when $D=\Z$, we just recover the usual Pr{\"u}fer rank $\Prank M$ of an abelian $p$-group $M$.)

The next result is a rewriting of~\cite[Theorem 4.1]{DC23}, specialised in the case of a radical ring.
\begin{theorem}\label{theorem: main del corso}
Let $p$ be a prime number, and let $D$ be a PID such that $pD$ is a prime ideal of $D$. Let $N$ be a radical ring of order a power of $p$ such that $N$ is a $D$-module. Assume that
\begin{equation*}
	\Prank_D N <p-1.
\end{equation*} 
Then $N$ and $1+N$ have the same number of elements of each order. In particular, if the ring $N$ is commutative, then $N\cong  1+N$.  
\end{theorem}

\subsection{The case of small radical rings}
In the following we give a converse of  Theorem \ref{theorem: main del corso}.
Let $D$ be a principal ideal domain, and let $p$ be a prime such that $pD$ is a prime ideal of $D$. Assume that $D/pD$ is finite, and write $\lambda=[D/pD:\F_p]$, so that $D/pD\cong \F_{p^\lambda}$.

Given a $D$-module $M$ of order a power of $p$, write $\Omega(M)$ for the \emph{socle} of $M$:
\begin{equation*}
	\Omega(M)=\{x\in M\mid px=0\}.
\end{equation*}
Note that $\Omega(M)$ is a $D$-module annihilated by $p$, hence an $\F_{p^\lambda}$-vector space.

\begin{lemma}\label{lemma: technical socle}
	Let $M$ be a $D$-module of order a power of $p$. Then
	\begin{equation*}
		\dim_{\F_{p^{\lambda}}}\Omega(M)=\Prank_D\Omega(M)=\Prank_D M.
	\end{equation*} 
\end{lemma}

\begin{proof}
The assumption on $D$ shows that $pD$ is the unique prime ideal of $D$ such that $pD\cap \Z=(p)$; therefore, applying the structure theorem for finitely generated modules over a PID, we find that we can write $M$ as product of indecomposable cyclic $D$-modules 
	\begin{equation*}
		M\cong \bigoplus_{i=1}^r D/p^{a_i}D,
	\end{equation*}
	for some $a_i>0$. 
	In particular, the group isomorphisms
	\begin{equation*}
		\Omega(M)\cong \bigoplus_{i=1}^r p^{a_i-1}D/p^{a_i}D\cong \bigoplus_{i=1}^r D/pD.
	\end{equation*}
 yield the desired equalities. 
\end{proof}

\begin{corollary}
	\label{corollary: lambda rank}
	Let $M$ be a $D$-module of order a power of $p$. Then
	\begin{equation*}
		\Prank M=\lambda\cdot \Prank_{D} M.
	\end{equation*}
\end{corollary}
\begin{proof}
	By Lemma~\ref{lemma: technical socle},
	\begin{equation*}
		\Prank M=\dim_{\F_p} \Omega(M)=\lambda\cdot \dim_{\F_{p^{\lambda}}} \Omega(M)=\lambda\cdot \Prank_{D} M.\qedhere
	\end{equation*}
\end{proof}
\begin{definition}
\label{def:lambdasmall}
Let $M$ be an abelian $p$-group, and let $\lambda$ be a  positive integer. We say that $M$ is \emph{$\lambda$-small} if $\Prank M <\lambda(p-1)$. 
\end{definition}
We arrive to our main result, in which we rewrite and provide a converse to Theorem~\ref{theorem: main del corso}.

\begin{theorem}\label{theorem: main}
Let $D$ be a principal ideal domain, and let $p$ be a prime such that $pD$ is a prime ideal of $D$ and $D/pD$ is finite; write $\lambda=[D/pD:\F_p]$.
	Let $N$ be a commutative radical ring of order a power of $p$ such that $N$ is a $D$-module. If one between $N$ and $1+N$ is $\lambda$-small, then $N\cong 1+N$.  
\end{theorem}

\begin{proof}
If we suppose that $N$ is $\lambda$-small, them result follows from Theorem~\ref{theorem: main del corso} and Corollary~\ref{corollary: lambda rank}

Conversely, suppose that $1+N$ is $\lambda$-small.	To lighten the notation, write $\dim$ for $\dim_{\F_p}$ and $M=\Omega(N)$. Then $M$ is an $\F_p$-vector space, and given $M^{(p)}=\{x^p\mid x\in M\}$ and $M_p=\{x\in M\mid x^p=0\}$, there exists an obvious isomorphism
	\begin{equation*}
		M/M_p\cong M^{(p)}
	\end{equation*}
	as $\F_p$-vector spaces, which implies
	\begin{equation*}
		\Prank N=\dim M=\dim M_p+\dim M^{(p)}.
	\end{equation*}
	Now note that $M_p$ is contained in $\Omega(1+N)$, by the well-known formula given in~\cite[Lemma 3]{FCC}. We derive that 
	\begin{equation*}
		\Prank N\le \Prank (1+N)+\dim M^{(p)}.
	\end{equation*}
When $M^{(p)}=0$, the inequality given implies that also $N$ is $\lambda$-small, as claimed. 

So suppose that $M^{(p)}\ne0$, hence also $M^p\ne 0$. (Here for all $i$, we denote by $M^i$ the ideal generated by the product of $i$ elements of $M$.) As $M$ is a finite radical ring, $M$ is nilpotent (see~\cite[Chapter IX,  Proposition 2.13]{hungerford}), so we have the chain of strict inequalities
\begin{equation*}
	\dim M^p< \dim M^{p-1}<\cdots<\dim M.
\end{equation*}
As each $M^i$ is an $\F_{p^{\lambda}}$-vector space, we deduce that
\begin{equation*}
	\dim M^p\le \dim M-\lambda(p-1),
\end{equation*}
and combining this with the previous equation, we conclude that
\begin{equation*}
	\Prank (1+N)\ge \lambda(p-1),
\end{equation*}
a contradiction. 
\end{proof}

\begin{remark}
	While in principle the previous result also holds for $p=2$, it is never applicable, as if $N$ is $\lambda$-small, then $\Prank N < \lambda$, that is, $\Prank_D N< 1$; this forces $N$ to be trivial.
\end{remark}

\begin{remark}
	This result can be applied in the framework of Hopf--Galois structures, to which radical rings and braces are tightly connected (see~\cite{FCC},~\cite{SV18}, and~\cite{ST23b}). Indeed, on the model of the main result of~\cite{FCC}, Theorem~\ref{theorem: main} proposes some circumstances under which the isomorphism class of the Galois group of a finite Galois extension and the type of a Hopf--Galois structure have to be the same.
\end{remark}


\section{Finite rings}
\label{sec:finiti}

Let $A$ be a finite commutative ring. As $A$ is Artinian, the structure theorem for Artinian rings says that we can write $A$ as the direct product of finite local rings.
In particular, in order to find the groups realisable in the class of finite commutative rings, we can focus our attention on finite commutative local rings. 

Given a finite local ring $A$, we write $\fm$ for its unique maximal ideal. Note that  the characteristic of $A$ is the power of a prime $p$, and  there exists $\lambda\ge 1$ such that
\begin{equation*}
	A/\fm\cong \F_{p^{\lambda}}.
\end{equation*} 
In this is the case, we write that $A$ is of $(p,\lambda)$-type. We will study the question of the realisability of  the groups in the class of finite commutative local rings of $(p,\lambda)$-type;  the finite abelian groups realisable in the class of finite rings will be the direct product of a finite number of them.

Let $A$ be a finite commutative local rings of $(p,\lambda)$-type. By~\cite[Theorem~3.1]{dcdAMPA}, 
\begin{equation}
\label{formula}
A^*\cong\F_{p^\lambda}^*\times1+\mathfrak m
\end{equation}
so the question reduces to find the possible structure of $1+\fm$, the adjoint group of the radical ring $\fm$.

The  question seems to be difficult in general. Some information can be derived from the knowledge of the group  structure of $\fm$, which can be studied with methods developed in \cite{DC23}. First, we state an useful lemma~\cite[Lemma 4.8]{DC23}. 
\begin{lemma}\label{lemma: restriction of scalars}
	Let $A$ be a finite commutative local ring of $(p,\lambda)$-type. Then every $A$-module is also a $\Z_p(\lambda)$-module by restriction of scalars.
\end{lemma}

\begin{corollary}\label{cor: m is lambda power}
	Let $p$ be a prime, and let $A$ be a finite commutative local ring of $(p,\lambda)$-type. Then $\mathfrak{m}$ is the $\lambda$-power of an abelian $p$-group.
\end{corollary}

\begin{proof}
	By Lemma~\ref{lemma: restriction of scalars}, $\fm$ is a finite $\Z_p(\lambda)$-module. Then the result follows from Proposition~\ref{proposition: a Zp-module is lambda power}.
\end{proof}

The last corollary gives us the information that  $1+\fm$ is an abelian $p$-group of order a power of $p^\lambda$. 
However, it turns out that $\fm$ and $1+\fm$ are not necessarily isomorphic  and,  for $\lambda>1$, the group  $1+\fm $ is not necessarily the $\lambda$-power of a $p$-group (see \cite[Example~2]{dcdAMPA}).
 On the other hand,  in view of the the constraints on the filtration of the group $1+\fm$ given in~\cite[Theorem~3.1]{dcdAMPA},  it can not be any $p$-group if $\lambda>1$,  nor any group of order a power of $p^\lambda$ (see \cite[Proposition~4.1 and Example~1]{dcdAMPA}).

Nevertheless, for $p$ odd, it is possible to construct examples of rings for which both $\fm$ and $1+\fm$ are isomorphic to the $\lambda$-power of a given abelian $p$-group. The following proposition is~\cite[Proposition 4.3]{dcdAMPA}. 
\begin{prop}\label{proposition: odd prime realisability}
Let $p$  be an odd prime. For all $\lambda\ge 1$ and for every finite abelian $p$-group $P$, there exists a finite local commutative ring $A$ of $(p,\lambda)$-type such that
	\begin{equation*}
		A^\ast\cong \F_{p^{\lambda}}^\ast \times P^{\lambda}.
	\end{equation*}
	\end{prop}
	
	The previous proposition shows that when $\lambda=1$ and $p$ is odd then $1+\fm$ can be any abelian $p$-group, completely solving the question of groups realising in this case.

\subsection{The case of small groups} 
In the case when   $\fm$ or $1+\fm$  is small, by applying Theorem~\ref{theorem: main} and Corollary~\ref{cor: m is lambda power} we can describe all the possibility for the group of units of $A$.

\begin{prop}\label{proposition: main finite}
	Let $A$ be a finite commutative local ring of $(p,\lambda)$-type. Suppose that one between $\fm$ and $1+\fm$ is $\lambda$-small. Then $\fm\cong 1+\fm$. In particular, $1+\fm$ is the $\lambda$-power of a $1$-small abelian $p$-group.
\end{prop}

\begin{proof}
	By Lemma~\ref{lemma: restriction of scalars}, $\fm$ is a $\Z_p(\lambda)$-module. In particular, we can apply Theorem~\ref{theorem: main} to prove that if one between $\fm$ and $1+\fm$ is $\lambda$-small, then $\fm\cong 1+\fm$. By Corollary~\ref{cor: m is lambda power}, we conclude that if this is the case, then $1+\mathfrak{m}$ is the $\lambda$-power of an abelian $p$-group.
\end{proof}

The last proposition allows us to prove that, under some conditions, the rings obtained via Proposition~\ref{proposition: odd prime realisability} are the only ones we can obtain, so we can give some answer to Fuchs' problem.

\begin{theorem}\label{theorem: main finite case}
	Let $\lambda$ be a positive integer, and let $p$ be an odd prime. Then the finite abelian groups with $\lambda$-small Sylow $p$-subgroup appearing as groups of units of finite local commutative rings of $(p,\lambda)$-type are exactly those of the form
	\begin{equation*}
		\F_{p^{\lambda}}^\ast \times P^{\lambda},
	\end{equation*}
	where $P$ is a $1$-small abelian $p$-group.
\end{theorem}

\begin{proof}
First, let $P$ be a  $1$-small abelian $p$-group. By Proposition~\ref{proposition: odd prime realisability}, there exists a finite commutative local ring $A$ of $(p,\lambda)$-type such that  
\begin{equation*}
		A^{\ast}\cong \F_{p^{\lambda}}^\ast \times P^{\lambda}.
\end{equation*}
Clearly the Sylow $p$-subgroup of $A^{\ast}$ is $\lambda$-small. 	
	
Conversely, suppose that a group $G$ with $\lambda$-small Sylow $p$-subgroup is realisable as group of units of a finite local commutative ring  with $A$ of $(p,\lambda)$-type. Then 
	\begin{equation*}
		G\cong A^\ast\cong \F_{p^{\lambda}}^\ast \times 1+\fm.
	\end{equation*}
	By assumption,  $1+\fm$ is $\lambda$-small; hence $\fm\cong 1+\fm$ is the $\lambda$-power of a  $1$-small abelian $p$-group by Proposition~\ref{proposition: main finite}.
\end{proof}

\begin{remark}
\label{rem:finitechar}
 By Corollary~\ref{cor:finiti}, the finitely generated abelian groups  that can be realised as the groups of unit of a finite characteristic ring are exactly those whose torsion subgroup is realisable. Therefore, the previous result gives families of both realisable and non-realisable finitely generated abelian groups of any rank, in the class of finite characteristic rings.
 \end{remark}

The following example shows how our result can be used to prove that some group  is, or is not, realisable in the class of (finite) rings.

\begin{example}
	Consider the group
	\begin{equation*}
		G=(\Z/5\Z)^2\times \Z/600\Z. 
	\end{equation*}
	We claim that $G$ is not realisable as the group of units of any ring. To show this, we first suppose that $G$ is realisable in the class of finite rings. We write
	\begin{equation*}
		G\cong  \Z/8\Z\times\Z/3\Z\times (\Z/5\Z)^2\times \Z/25\Z,
	\end{equation*}
	and assume, by contradiction, that 
	 $G=A^*$ for some finite ring $A$. Let $A\cong A_1\times\dots\times A_s$, where each  $A_i$ is  local of $(p_i,\lambda_i)$-type. Then
	 \begin{equation*}
	 	G\cong \prod_{i=1}^s(\F_{p_i^{\lambda_i}}^*\times 1+\mathfrak m_i).
	 \end{equation*}
When $p$ is odd  $\F_{p^\lambda}^*$ has even cardinality, so, having $G$ a cyclic Sylow, $2$-subgroup  at most one of the $p_i$ can be  odd.

Looking at the Sylow $5$-subgroup  of $G$ we contend that one of the $p_i$ must be 5.
In fact, otherwise there must be three primes $p_i$ such that 5 divides $|\F_{p_i^{\lambda_i}}^*|$, and at least two of them, say $p_1$ and $p_2$, need to be the prime 2. This means that $2^{\lambda_1}-1$ and 
$2^{\lambda_2}-1$ should divide $600$, and therefore $75$, and be divisible by $5$, and it is immediate to see that the only possibility  is $2^{\lambda_1}-1=2^{\lambda_2}-1=15$; this is not possible since $G$ has just one factor of order 3.

So, $G$ should have a factor of the form
	\begin{equation*}
		\F_{5^{\lambda}}^*\times H,
	\end{equation*}
	where $H$ is a $5$ group. This forces $\lambda=2$ and $H$ to be a factor of
	\begin{equation*}
		(\Z/5\Z)^2\times\Z/25\Z.
	\end{equation*}
In particular, 
		\begin{equation*}
		\Prank H\le 3<\lambda_5(5-1). 
	\end{equation*}
We can apply Theorem~\ref{theorem: main finite case} to deduce that $H$ is the $2$-power of an abelian $5$-group, that is, $H$ is trivial or $H=(\Z/5\Z)^2$. 

But the $\Z/25\Z$ cannot be a factor of the group $\F_{2^{\lambda}}^*$, so $G$ is not realisable by finite rings.
	
Finally, we note that by Proposition~\ref{prop:k<2} below, no factor of $G$ may be realised as the group of units of a TN ring, deriving our claim. 
\end{example}

\section{TN rings: The tools}\label{section: tn tools}

 The aim of this section is to introduce the tools for the study of the group of units  of  TN rings. This machinery will  then be used in the next section to classify, within a certain class of \emph{small} finitely generated abelian groups,  those that are realisable in the class of TN ring. 
 
 We begin our analysis with torsion-free rings, as if $A$ is a $\TN$ ring, then $A/\nt$ is a torsion-free ring, and its group of units appears in the short exact sequence~\eqref{eq:success}. 
  
 \subsection{A review on torsion-free rings}
 
We collect here some facts about torsion-free rings proved in \cite{JLMS}. These results will be useful for the discussion of the general case. 

The following theorem is \cite[Theorem~5.1]{JLMS}.
\begin{theorem}\label{teo:thm5.1JLMS}
Let $T$ be a finite abelian group of even order. 	The group $T\times \Z^r$ is realisable in the class of torsion-free rings
if and only if $r\ge g(T)$, where $g(T)$ is a constant depending on the structure of $T$. 
\end{theorem}	
The constant $g(T)$ is specified in~\cite[Equation~(12)]{JLMS}.

We now restrict our study to groups $T$ with a cyclic Sylow $2$-subgroup. Consider the decomposition of $T$ into cyclic groups of order the power of a prime
\begin{equation}
\label{eq:T}
	T\cong \Z/2^{\varepsilon}\Z \times \prod_{i=1}^s \Z/p_i^{a_i}\Z,
\end{equation}
where $\varepsilon\ge1$, $s\ge 0$, and $p_i$ is an odd prime with $a_i\ge 1$ for all $i$. Write $s_0$ for the cardinality of the set $\{p_1,\ldots,p_s\}$. The value of  $g(T)$ in case is 
\begin{equation*}
	g(T)=\sum_{i=1}^s \left(\frac{\phi(2^{\varepsilon}p_i^{a_i})}{2}-1\right)+c(T),
\end{equation*}
where 
\begin{equation*}
	c(T)=\begin{cases}
		\left(\frac{\phi(2^{\varepsilon})}{2}-1\right)^{*}&\text{if $s_0\ne 1$},\\
		0 & \text{if $s_0= 1$}.
	\end{cases}
\end{equation*}
Here $\phi$ denotes the Euler's totient function and $\left(\frac{\phi(2^{\varepsilon})}{2}-1\right)^{*}$ denotes $\frac{\phi(2^{\varepsilon})}{2}-1$ except when $\varepsilon=1$, in which case it is equal to $0$ by convention. 

Note that, by the description of $g(T)$, if $T'$ is a subgroup of $T$ of even order, then 
\begin{equation}\label{eq:g-crescente}
	g(T')\le g(T)
\end{equation}
(this is no more true when the Sylow $2$-subgroup of $T$ is not cyclic).

Following \cite{JLMS}, we present explicitly two class of torsion-free ring useful to realise minimum values in the next sections.  In the following example we keep the notation given in~\eqref{eq:T}.
\begin{example}\label{example: torsion-free not 1}
	Let
\begin{equation*}
	\mathcal{M}=\prod_{i=1}^s \Z[\zeta_{2^{\varepsilon}p_i^{a_i}}] \times \Z[\zeta_{2^{\varepsilon}}]. 
\end{equation*}
For all $i=1,\ldots,s$,  take $\beta_i=(1,\ldots, 1,\zeta_{p_i^{a_i}},1,\ldots,1)$, where  $\zeta_{p_i^{a_i}}$ appears in the $i$th component, and define $B=\Z[\zeta_{2^{\varepsilon}}][\beta_1,\ldots,\beta_s]$ (here we are identifying $B$ with a subring of $\mathcal{M}$ via the diagonal embedding of $\Z[\zeta_{2^\varepsilon}]$). This is a reduced torsion-free ring.

As clarified in~\cite[Subsection 5.2]{JLMS}, the following holds:
\begin{equation*}
	B^{\ast}\cong
	\begin{cases}
		T\times \Z^{g(T)} &\text{if $s_0\ne 1$},\\
		T\times \Z^{g(T)+g(\Z/2^{\varepsilon}\Z)} &\text{if $s_0= 1$}. 
	\end{cases}
\end{equation*}
In particular, in the case $s_0 \ne 1$, the ring $B$ realises the minimum value of the rank of a torsion-free ring whose group of torsion units as in \eqref{eq:T}.
\end{example}

\begin{example}\label{example: torsion-free 1}
	Suppose that $s_0=1$, so that 
	\begin{equation}
	\label{eq:s0=1}
		T\cong \Z/2^{\varepsilon}\Z\times \prod_{i=1}^v \Z/p^{a_i}\Z,
	\end{equation}
	with $1\le a_1\le\cdots\le a_v$. Let
\begin{equation*}
	\mathcal M=\prod_{i=1}^v \Z[\zeta_{2^{\varepsilon}p^{a_i}}]. 
\end{equation*}
For all $i=2,\ldots,v$,  take $\beta_i=(1,\ldots, 1,\zeta_{p^{a_i}}1,\ldots,1)$, where  $\zeta_{p^{a_i}}$ appears in the $i$th component, and define $B=\Z[\zeta_{2^{\varepsilon}p^{a_1}}][\beta_2,\ldots,\beta_v]$ (here we are identifying $B$ with a subring of $\mathcal M$ via the diagonal embedding of $\Z[\zeta_{2^\varepsilon p^{a_1}}]$). This is a reduced torsion-free ring. Then as clarified in~\cite[Subsection 5.2]{JLMS}
\begin{equation*}
	B^{\ast}\cong T\times \Z^{g(T)}. 
\end{equation*}
In particular, this ring realises the minimum value of the rank of a torsion-free ring whose group of torsion units $T$ is as in \eqref{eq:s0=1}. 
\end{example}

\begin{remark}\label{remark: these are connected}
	Recall that a commutative ring is \emph{connected} if it has no idempotents different from $0$ and $1$. We claim that any torsion-free ring  with a unique element of multiplicative order $2$ is connected, so this is true for the rings obtained in Examples~\ref{example: torsion-free not 1} and~\ref{example: torsion-free 1}.
	
	In fact, let by $\epsilon$ an idempotent of $B$. Then $(1-2\epsilon)^2=1$, and therefore
	\begin{equation*}
		1-2\epsilon\in\{1, -1\}.
	\end{equation*} 
	This forces $2\epsilon=0$ or $2(\epsilon-1)=0$. Since $B$ is torsion-free, we can concldue that $\epsilon=0$ or $\epsilon=1$, as claimed. 
\end{remark}

If $B$ is a connected ring, then the group of units of the ring of Laurent polynomials $B[x, x^{-1}]$ is  the group $\langle B^*, x\rangle\cong B^*\times \Z$. This construction preserves the property of the ring of being connected and also of being  TN. Therefore, an immediate inductive argument  gives the next proposition, which allows us to  \say{increase the rank} of the group of units while remaining in the same class of rings.

 \begin{prop}\label{prop: laurent}
 	Let $B$ be a torsion-free reduced connected ring, and let $r\ge 1$. Write
 	\begin{equation*}
 		C=B[x_1,\ldots,x_r,x_1^{-1},\ldots x_r^{-1}].
 	\end{equation*}
 	Then
 	\begin{equation*}
 		C^{*}\cong B^*\times \Z^r.
 	\end{equation*}
 \end{prop}

For a next use, we recall the following classification result for finite abelian groups; see~\cite[Theorem B]{dcdBLMS}.

\begin{theorem}\label{theorem: finite torsion-free}
	The finite abelian groups realisable in the class of torsion-free rings are exactly those of the form
	\begin{equation*}
		(\Z/2\Z)^a\times (\Z/4\Z)^b\times (\Z/3\Z)^c,
	\end{equation*}
	where  $a,b,c\in \N$, $a+b\ge 1$, and $a\ge 1$ if $c\ge 1$.
	\end{theorem}

 \subsection{The role of $\varepsilon$}
 
Recall that TN rings have characteristic zero, and therefore  $-1$ is a unit of order of order $2$. In addition, as TN rings $A$ are assumed to have $A^*$ finitely generated, $(A^*)_{\tors}$ has even order.

For any finitely generated abelian group $G$, we denote by  $\varepsilon(G)$  the minimum $k$ such that $G_{\tors}$ has a cyclic factor of order $2^k$ in its decomposition into cyclic groups of prime power order.  The value of $\varepsilon(A^*)$ is a significant invariant for the structure of the ring and especially for its group of units, as we will see in what follows. 

For $n\ge 1$, say that a TN ring $A$ is an \emph{$n$-ring} if
 $A$ has a subring isomorphic to the cyclotomic ring $\Z[\zeta_n]$. (When $A$ is an $n$-ring, we identify $\Z[\zeta_n]$ with its copy in $A$.) The concept of $n$-ring naturally aligns with the concept of $\Z[\zeta_n]$-algebra.

\begin{prop}
Let  $A$ be a TN ring. The following are equivalent:
\begin{enumerate}
\item $A$ is a $\Z[\zeta_n]$-algebra.
\item $A$ is an $n$-ring
\end{enumerate}
\end{prop}
 \begin{proof}
 $(2)\Rightarrow(1)$ is clear. Conversely, assume that $A$ is a $\Z[\zeta_n]$-algebra, and let 
 \begin{equation}
 \label{eq:psi}
 \psi\colon	\Z[\zeta_n]\to A,\quad z\mapsto z1
 \end{equation}  
be the unit map. This map is injective, because if it has a non-trivial kernel, then $1$ would be a torsion element, as all the quotients of $\Z[\zeta_n]$ by non-trivial ideals are finite by Proposition~\ref{proposition: cyclotomic quotients}.
 \end{proof}

The proof of the next proposition  follows the same lines as that of \cite[Theorem~5.1]{dcdBLMS}.
\begin{prop}
\label{prop:epsilon-1}
Let  $A$ be a TN ring. The following are equivalent:
\begin{enumerate}
\item $A$ is a $2^\varepsilon$-ring.
\item There exists $\alpha\in A$ such that $-1=\alpha^{2^{\varepsilon-1}}$.
\end{enumerate}
\end{prop}
\begin{proof}
Assume that $A$ is a ${2^\varepsilon}$-ring, and let $\psi$ be the map in~\eqref{eq:psi}. Then for $\alpha=\psi(\zeta_{2^\varepsilon})$,  we have $\alpha^{2^{\epsilon-1}}=-1$.

Assume now that (2) holds, and let
$\varphi_\alpha\colon \Z[x]\to A$ be defined by $p(x)=p(\alpha)$. The  polynomial $\mu(x)=x^{2^{\varepsilon-1}}+1$ belongs to the kernel of $\varphi_\alpha$. Now, $\mu(x)$ is the cyclotomic polynomial of the of the $2^\varepsilon$th roots of units, so it is irreducible. By~\cite[Lemma 5.2]{dcdBLMS}, we obtain that $\ker(\varphi_\alpha)=(\mu(x))$, so $\Z[\zeta_{2^{\varepsilon}}]\cong\Z[x]/(\mu(x))\hookrightarrow A$.
\end{proof}

 \begin{corollary}
\label{cor:epsilon-group}
Let  $A$ be a TN ring. 
If $\varepsilon(A^*)=\varepsilon$, then $A$ is a $2^\varepsilon$-ring.
\end{corollary}
\begin{proof}
In a finitely generated abelian group $G$ with $\varepsilon(G)=\varepsilon$, every element of order $2$ is a $2^{\varepsilon-1}$-power. This implies that $-1$ is necessarily a $2^{\varepsilon-1}$-power in $A^*$, and the result follows from Proposition~\ref{prop:epsilon-1}.
\end{proof}

 \subsection{The short exact sequence}
  We now analyse the information we can get from the exact sequence of Proposition~\ref{prop:successioneesatta}, when $A$ is a TN ring and $\mathfrak I=\nt$:
 \begin{equation}
\label{eq:es-nt}
1\to 1+\nt\hookrightarrow A^*{\to}\left(A/\nt\right)^*\to 1.
\end{equation}
As already mentioned, this choice is particularly interesting since as  $A$ is TN, the torsion ideal of $A$ is contained into $\mathfrak N$, whence in its torsion part $\nt$. Therefore, $A/\nt$ is a torsion-free ring and the finitely generated abelian groups realisable in this case  are known by \cite{JLMS}. 

We begin by presenting a consequence of Proposition~\ref{prop:epsilon-1}. 

\begin{corollary}
\label{cor:B2allaepsilon}
Let $A$ be a TN ring. If $A$ is a $2^\varepsilon$-ring, then $A/\nt$ is a  $2^\varepsilon$-ring.
\end{corollary}  

\begin{proof}
By Proposition~\ref{prop:epsilon-1} there exists $\alpha\in A$ such that 
$-1=\alpha^{2^{k-1}}$, then the same equality holds modulo $\nt$, so $A/\nt$ is a  $2^\varepsilon$-ring. 
\end{proof}

\begin{remark}
	The previous corollary also shows that the copy of $\Z[\zeta_{2^\varepsilon}]$ contained in $A$ projects to a copy of $\Z[\zeta_{2^\varepsilon}]$ under the natural projection
	\begin{equation*}
		A\to A/\nt.
	\end{equation*}
\end{remark}

We derive a result in relation to finite group of units.

\begin{prop}\label{prop:k<2}
Let  $A$ be a TN ring such that $A^*$ is finite. Then $\varepsilon(A^*)\le 2$. In particular, if $-1$ is a $2^k$-power, then $k\le 1$.  
\end{prop}

\begin{proof}
Write $\varepsilon(A^*)= \varepsilon$. By Corollaries~\ref{cor:epsilon-group} and~\ref{cor:B2allaepsilon}, we deduce that the ring $\Z[\zeta_{2^{\varepsilon}}]$ injects in $B=A/\nt$; this injection can be restricted to the groups of units. Now, $B^*\cong A^*/1+\nt$ is finite, so  $\Z[\zeta_{2^{k+1}}]^*$ is also finite, and  the Dirichlet's Units Theorem (see~\cite[Theorem 37]{FT93}) implies that  $\varepsilon(A^*)\le 2$. The final assertion follows from Proposition~\ref{prop:epsilon-1}.
\end{proof}

The other group appearing in \eqref{eq:es-nt} is $1+\nt$, which is the adjoint group of the radical ring $\nt$. In some cases,   the methods developed in Section~\ref{section: radical} allows us to determine its structure. In order to do so, we need to restrict the attention to some particular characteristic zero (and TN) rings. In light of Proposition~\ref{prop:ABC}, this will not be a limitation in what follows.

\begin{prop}\label{prop:rango}

Let $A$ be a TN ring that is finitely generated and integral over $\Z$. The following hold:
\begin{enumerate}
\item $\nt$ is a finite radical ring.
\item $\rank A^*=\rank (A/\nt)^*$.
\end{enumerate}
 \end{prop}
\begin{proof}
As $A$ is finitely generated and integral over $\Z$, it is also finitely generated as $\Z$-module, and the same holds for its subgroup $\n$. It follows that its torsion part $\nt$ is finite. In addition,  $\nt$ is an ideal of $A$ consisting of nilpotent elements, therefore it is a finite radical ring. In particular, $1+\nt\subseteq A^*_{\tors}$ is a finite abelian group, and by
 \begin{equation*}
\left(\frac{A}{\nt}\right)^*\cong \frac{A^*}{1+\nt}\cong \frac{A^*_{\tors}}{1+\nt}\times \Z^r,
\end{equation*}
 where $r$ is the rank of the group $A^*$, we obtain that
\begin{equation*}
	\rank A^*=\rank(A/\nt)^*. \qedhere
\end{equation*}
\end{proof}

The same argument of the proof of the previous proposition, in the same hypothesis, shows the exactness of the sequence 
\begin{equation}
\label{eq:se-tors}
1\to 1+\nt\hookrightarrow A^*_{\tors}{\to}\left(A/\nt\right)^*_{\tors}\to 1.
\end{equation}

\begin{notation}
	Given a prime $p$, we write $T_p$ for the Sylow $p$-subgroup of a finite abelian group $T$.
\end{notation}
 Consider the Sylow $p$-subgroup $\ntp$ of $\nt$. It is easy to check that $\ntp$ is an  ideal of  $\nt$; it follows that $\ntp$ is again a radical ring, so $1+{\ntp}$ is a subgroup of $1+\nt$ and, by cardinality, we deduce that $(1+\nt)_p=1+\ntp$. In particular, also the sequence
\begin{equation}
\label{eq:se-sylow}
1\to 1+{\ntp}\hookrightarrow A^*_{\tors,p}{\to}\left(A/\nt\right)^*_{\tors,p}\to 1
\end{equation}
is exact. 
We want now to study  the structure of $1+\ntp$, in the case when Theorem~\ref{theorem: main} applies. First, we mention a consequence of~\cite[Proposition 4.10]{DC23}.

\begin{lemma}\label{lemma: restriction of scalars tn}
		Let $A$ be a TN ring, and write $\varepsilon(A^*)=\varepsilon$. Then every finite $A$-module of order a power of a prime $p$ is also a $\Z_p(\lambda(p,2^{\varepsilon}))$-module.
\end{lemma}

\begin{proof}
	Let $N$ be a finite $A$-module of order a power of $p$. If $\varepsilon(A^*)=\varepsilon$, then $A$ is a $2^{\varepsilon}$-ring by Corollary~\ref{cor:epsilon-group}. In particular, $N$ is a $\Z[\zeta_{2^{\varepsilon}}]$-module. Denote by $I$ the annihilator of $N$ in $\Z[\zeta_{2^{\varepsilon}}]$, so that $N$ is a faithful $\Z[\zeta_{2^{\varepsilon}}]/I$-module. By Proposition~\ref{proposition: cyclotomic quotients},  
	\begin{equation*}
		\frac{\Z[\zeta_{2^\varepsilon}]}{I}\cong \prod_{i=1}^n \frac{\Z[\zeta_{2^{\varepsilon}}]}{P_i^{a_i}},
	\end{equation*}
	where each term is a finite commutative local ring of type $(p,\lambda(p,2^{\varepsilon}))$. We are in position to apply~\cite[Proposition 4.10]{DC23}, which yields that $N$ is a $\Z_p(\lambda(p,2^{\varepsilon}))$-module. 
\end{proof}

This yields a  description of  the structure of $\ntp$.

\begin{corollary}\label{corollary: module brace and tn}
Let $A$ be a TN ring that is finitely generated and integral over $\Z$, and write $\varepsilon(A^*)=\varepsilon$. If $p$ is an odd prime, then
 $\ntp$ is the $\lambda(p,2^{\varepsilon})$-power of an abelian $p$-group.
\end{corollary} 

\begin{proof}
By Proposition~\ref{prop:rango} and Lemma~\ref{lemma: restriction of scalars tn},  the radical ring $\ntp$ is a finite module over  $\Z_p(\lambda(p,2^{\varepsilon}))$ of order a power of $p$, so we can apply Proposition~\ref{proposition: a Zp-module is lambda power} to derive the assertion.\end{proof}

We arrive to the desired result.

\begin{prop}\label{proposition: tn small}
Let $A$ be a TN ring that is finitely generated and integral over $\Z$, and write $\varepsilon(A^*)=\varepsilon$.  Let $p$ be an odd prime such that one between $\ntp$ and $1+\ntp$ is $\lambda(p,2^{\varepsilon})$-small. Then $\ntp\cong 1+\ntp$. In particular, $1+\ntp$ is the $\lambda(p,2^{\varepsilon})$-power of an abelian $p$-group.
\end{prop}

\begin{proof}
	By Proposition~\ref{prop:rango} and Lemma~\ref{lemma: restriction of scalars tn},  the radical ring $\ntp$ is a finite module over  $\Z_p(\lambda(p,2^{\varepsilon}))$ of order a power of $p$. In particular, we can apply Theorem~\ref{theorem: main} to prove that if one between $\fm$ and $1+\fm$ is $\lambda(p,2^{\varepsilon})$-small, then $\fm\cong 1+\fm$. By Corollary~\ref{corollary: module brace and tn}, we conclude that if this is the case, then $1+\mathfrak{m}$ is the $\lambda(p,2^{\varepsilon})$-power of an abelian $p$-group.
\end{proof}

\section{Realisability in the class of TN rings}\label{section: tn units}
The results of the previous section can be used to characterise the finitely generated abelian groups, with some additional condition on the torsion part, that can be realised in the class of TN rings. 

As recalled in Theorem~\ref{teo:thm5.1JLMS},  every finite abelian group of even cardinality $T$ is  realisable as the torsion part of  a finitely generated abelian group of units already within the class of torsion-free rings. In the torsion-free case, the structure of $T$ gives a lower bound $g(T)$ for the possible rank of the group of units, and our aim is to understand---at least in some cases---which is the minimum rank $r(T)$ for which the group 
$$T\times \Z^{r(T)}$$ 
is realisable in the wider class of TN rings.
Clearly $r(T)\le g(T)$. 

In order to find the value of $r(T)$, on the one hand, we need  to show that every TN ring $A$ with $A^*\cong T\times \Z^r$ satisfies 
\begin{equation*}
	r\ge r(T),
\end{equation*}
on the other hand, we have to construct a ring $A$ with 
\begin{equation*}
	A^*\cong T\times \Z^{r(T)}.
\end{equation*}
We proceed with the former in the next subsection, thanks to our result controlling the shape of $1+\nt$, given a suitable TN ring $A$.

The latter is instead taken care of in the second subsection, where we present a construction not only to realise $T\times \Z^{r(T)}$, but also $T\times \Z^{r(T)+l}$ for all $l$, meaning that $T\times \Z^{r}$ is realisable in the class of TN rings if and only if $r\ge r(T)$.

As we mentioned before, our results concern a restricted class of abelian groups $T$, that we now describe.
 First, we assume that $T_2$ is cyclic, say $T_2\cong \Z/2^{\varepsilon}\Z$ with $\varepsilon\ge 1$. Now write $\lambda_{\ell}=\lambda(\ell,2^{\varepsilon})$ for all odd primes $\ell$. The second assumption is that if $p_1,\ldots, p_{s_0}$ are the distinct odd prime divisors of $|T|$ such that $T_{p_i}$ is not a $\lambda_{p_i}$-power, then 
\begin{equation*}
	\Prank T_{p_i}<\lambda_{p_i}. 
\end{equation*}
(This assumption is always satisfied if $T$ is cyclic.) If $q_0,\ldots,q_{t_0}$ are the remaining distinct odd prime divisors of $T$, then we can write $T_{q_j}=(V_{q_j})^{\lambda_{q_j}}$, and therefore
\begin{equation*}
	T\cong \Z/2^\varepsilon\Z\times \prod_{i=1}^{s_0} T_{p_i} \times \prod_{j=0}^{t_0} (V_{q_j})^{\lambda_{q_j}}
\end{equation*}
When $s_0=1$, we simply write $p_1=p$ and
\begin{equation*}
	T_p\cong \prod_{i=1}^{v}\Z/p^{a_i}\Z,
\end{equation*}
where $1\le a_1\le\cdots\le a_v$. 

We denote this class of groups as $\geps$ (recall that $\varepsilon \ge 1$), and for groups of this class we will always use the notation given above.

\subsection{Bounding from below}

We begin by presenting a first lower bound.
\begin{prop}\label{prop:lb1}
	Let $T\in\geps$. Then 
	\begin{equation*}
	r(T)\ge g\left(\Z/2^{\varepsilon}\Z \times \prod_{i=1}^{s_0} T_{p_i}\right).
\end{equation*}
\end{prop}
\begin{proof}
Let $A$ be a TN ring such that $(A^*)_{\tors}\cong T$. As we are interested in a lower bound, because of Proposition~\ref{prop:ABC}, we can assume $A$ to be of the form $A=\Z[(A^*)_{\tors}]$. Write $B=A/\nt$, and we know that  both $A$ and $B=A/\nt$ are $2^{\varepsilon}$-rings. In particular, $B$ contains (a copy of) $\Z[\zeta_{2^{\varepsilon}}]$ as a subring. In addition, as
\begin{equation*}
	\Prank \n_{\tors,p_i}< \Prank T_{p_i}< \lambda_{p_i}< \lambda_{p_i}(p-1),
\end{equation*}
we can apply Proposition~\ref{proposition: tn small} to find that $1+\n_{\tors,p_i}$ is a $\lambda_{p_i}$-power, and therefore necessarily $1+\n_{\tors,p_i}$ is trivial for all $i$.
Therefore by using the short exact sequence~\eqref{eq:se-tors}  we obtain that
\begin{equation*}
	\Z/2^{\varepsilon}\Z \times \prod_{i=1}^{s_0} T_{p_i}
\end{equation*}
injects in $B^*_{\tors}$, resulting in 
\begin{equation*}
	\rank A^*=\rank B^*\ge g(B^{*}_{\tors})\ge  g\left(\Z/2^{\varepsilon}\Z \times \prod_{i=1}^{s_0} T_{p_i}\right).
\end{equation*}
As this holds for all $A$ with $A^{*}_{\tors}\cong T$, we derive 
\begin{equation*}
	r(T)\ge g\left(\Z/2^{\varepsilon}\Z \times \prod_{i=1}^{s_0} T_{p_i}\right).\qedhere
\end{equation*}
\end{proof}

For most of the cases, this bound will be sufficient to find the value of $r(T)$. We need now to take care of a \say{sporadic} case.

\begin{prop}\label{prop: s0=1 second}
	Let $T\in\geps$. Assume also that $s_0=1$, and that there exists an odd prime  $q\ne p$ such that $T_{q}$ is a not a $\lambda(q,2^{\varepsilon}p^{a_1})$-power. Then
	\begin{equation*}
		r(T)\ge g\left(\Z/2^{\varepsilon}\Z \times T_p\right)+g(\Z/2^{\varepsilon}\Z).
	\end{equation*}
\end{prop}

\begin{proof}
Write $\rho=g\left(\Z/2^{\varepsilon}\Z \times T_p\right)+g(\Z/2^{\varepsilon}\Z)$. Note that for $\varepsilon\in\{1,2\}$,
\begin{equation*}
	\rho=g\left(\Z/2^{\varepsilon}\Z \times T_p\right),
\end{equation*}
so there is nothing to prove. Suppose now that $\varepsilon\ge 3$. We assume that there exists a TN ring $A$ such that
	\begin{equation*}
		A^{*}\cong T\times \Z^r
	\end{equation*}
	with $r< \rho$. Replacing $A$ with $\Z[(A^*)_{\tors}]$, we can assume that $A$ is integral and finitely generated by units of finite order over $\Z$, as in Proposition~\ref{prop:ABC}. 
	The same therefore holds for the torsion-free ring  $B=A/\nt$. We know that $B^{\ast}_{\tors}$ contains (a copy of) $\Z/2^{\varepsilon}\Z \times T_p$. Suppose that $B^{\ast}_{\tors}$ is strictly bigger. Then there exists an odd prime $q$ such that 
	\begin{equation*}
	\Z/2^{\varepsilon}\Z \times T_p\times \Z/q\Z \hookrightarrow	B^{\ast}_{\tors}
	\end{equation*} 
	This is a contradiction, as
	\begin{equation*}
		g\left (\Z/2^{\varepsilon}\Z \times T_p
		\times \Z/q\Z\right)>g\left (\Z/2^{\varepsilon}\Z \times T_p
		\right)+g\left (\Z/2^{\varepsilon}\Z\right).
	\end{equation*}
Therefore,
	\begin{equation*}
B^*_{\tors}\cong \Z/2^{\varepsilon}\Z \times  T_p, \qquad		1+\nt\cong \prod_{j=1}^{t_0} (V_{q_j})^{\lambda_{q_i}},
	\end{equation*}
and the short exact sequence
 \begin{equation*}
 	1\to 1+\nt\to A^{\ast}_{\tors}\to B^{\ast}_{\tors}\to 1
 \end{equation*}
 splits. 

 Now we claim that $B$ is a $2^{\varepsilon}p^{a_1}$-ring, where, $p^{a_1}$ is the order of the smaller cyclic component of $T_p$, as in the given notation. This follows from the results of \cite{JLMS}. Indeed, the torsion free ring $B$ fulfil the assumption of \cite[Proposition~4.2]{JLMS}, therefore, by the subsequent Remark~4.3,
$B$ injects in a product of cyclotomic rings $\mathcal{M}$, say, which in turn must be $B^{*}_{\tors}$-\emph{admissible} in the sense of~\cite[Section 5.1]{JLMS}.

This means that $\mathcal{M}$ has the form
\begin{equation*}
	\mathcal{M}=\prod_{i=1}^{v}\Z[\zeta_{2^{\varepsilon}p^{a_i}u_i}]\times \prod_{i=1}^{w}\Z[\zeta_{2^{\varepsilon}u'_{i}}]
\end{equation*}
for some $u_i, u'_i$.
By~\cite[Lemma 4.4]{JLMS}, 
\begin{equation*}
	\rank\mathcal{M}^*=\rank B^*=r<\rho.
\end{equation*}
This yields that  $w=0$ and $u_i=1$ for all $i$, because otherwise
 \begin{equation*}
 	\rank \mathcal{M}^*=\sum_{i=1}^{v}\frac{\phi(2^{\varepsilon}p^{a_i}u_i)-1}2 + \sum_{i=1}^{w}\frac{\phi(2^{\varepsilon}v_i)-1}2\ge \rho>\rank B^*.
 \end{equation*}
 In particular, 
 \begin{equation*}
 	\mathcal{M}=\prod_{i=1}^{v}\Z[\zeta_{2^{\varepsilon}p^{a_i}}],
 \end{equation*}
 and this implies that $\mathcal{M}^{*}_{\tors,p}\cong B^{*}_{\tors,p}$. Now note that $\mathcal{M}$ contains a subring isomorphic to $\Z[\zeta_{p^{a_1}}]$, via diagonal immersion. We deduce that the element $(\zeta_{p^{a_1}},\ldots,\zeta_{p^{a_1}})$, which has order $p^{a_1}$, is in the image of the injection
 \begin{equation*}
 	\iota\colon B\hookrightarrow\mathcal{M},
 \end{equation*}
 and therefore, taking the preimage of the copy of $\Z[\zeta_{p^{a_1}}]$ in $\mathcal{M}$, we find that $B$ also contains, as a subring, a copy of $\Z[\zeta_{p^{a_1}}]$ (and thus of $\Z[\zeta_{2^{\varepsilon}p^{a_1}}]$).

Let us call $\beta'$ the preimage of $(\zeta_{p^{a_1}},\ldots,\zeta_{p^{a_1}})$ under $\iota$, and $\beta$ the element of $A^{*}_{\tors}$ mapping to $\beta'$ via the natural projection. The map

	\begin{equation*}
	\varphi'\colon \Z[\zeta_{2^{\varepsilon}}][x]\to B,\quad x\mapsto \beta'
\end{equation*}
has therefore kernel $(\Phi_{p^{a_1}}(x))$, the ideal generated by the  $p^{a_1}$-cyclotomic polynomial. Consider now
\begin{equation*}
	\varphi\colon \Z[\zeta_{2^{\varepsilon}}][x]\to A,\quad x\mapsto \beta.
\end{equation*}
We claim that $\ker \varphi=(\Phi_{p^{a_1}}(x))$. In this case,  $A$ would be a $2^{\varepsilon}p^{a_1}$-ring, so  $\nt$ would be a module over $\Z[\zeta_{2^{\varepsilon}p^{a_1}}]$.  Since by assumption $T_q$ is not a $\lambda(q,2^{\varepsilon}p^{a_1})$-power, then it can not be a subgroup of $1+\nt$, giving a contaddiction (recall Proposition~\ref{proposition: cyclotomic quotients}).

We just need to prove the claim. First, observe that as $\beta$ is an element of multiplicative order $p^{a_1}$, we have $x^{p^{a_1}}-1\in\ker\varphi$. Second, note that there exists a commutative diagram
\begin{equation*}
	\begin{tikzcd}
	{\Z[\zeta_{2^{\varepsilon}}][x]} & A \\
	& B
	\arrow["{\varphi'}"', from=1-1, to=2-2]
	\arrow["\varphi", from=1-1, to=1-2]
	\arrow["\pi", from=1-2, to=2-2]
\end{tikzcd}
\end{equation*}
where $\pi\colon A\to A/\nt=B$ is the obvious projection. This implies that
\begin{equation*}
	\ker \varphi \subseteq \varphi^{-1}(\nt)= \ker \varphi'= (\Phi_{p^{a_1}}(x)).
\end{equation*} 
We just need to show that $\varphi(\Phi_{p^{a_1}}(x))=0$. We know that
\begin{equation*}
	\varphi(\Phi_{p^{a_1}}(x))\in \nt,
\end{equation*} 
which means that $d\Phi_{p^{a_1}}(x)\in \ker\varphi$ and $\Phi_{p^{a_1}}(x)^c \in \ker\varphi$, where $d$ is the order of $\nt$ (which is coprime to $p$) and $c\ge 1$.

If $c=1$, the claim follows. If $c\ge 2$, then we can employ that fact that $x^{p^{a_1}}-1\in \ker \varphi$
and the equality 
\begin{align*}
	x^{p^{a_1}}-1=(x^{p^{a_1-1}}-1)\Phi_{p^{a_1}}(x)
\end{align*}
to derive that 
\begin{equation*}
	\Phi_{p^{a_1}}(x)(x^{p^{a_1-1}}-1,\Phi_{p^{a_1}}^{c-1}(x),d)\subseteq \ker\varphi.
\end{equation*}
But now $\Phi_{p^{a_1}}(x)^{c-1}=\Phi_{p}(x^{p^{a_1-1}})^{c-1}$, so
\begin{align*}
	(x^{p^{a_1-1}}-1,\Phi_{p^{a_1}}(x)^{c-1},d)&=(x^{p^{a_1-1}}-1,\Phi_{p}(x^{p^{a_1-1}})^{c-1},d)\\&=(x^{p^{a_1-1}}-1,\Phi_{p}(1)^{c-1},d)=(x^{p^{a_1-1}}-1,p^{c-1},d)=(1),
\end{align*}
as $p$ and $d$ are coprime; this means that $\Phi_{p^{a_1}}(x)\in \ker\varphi$, as claimed. 
\end{proof}

\subsection{Realising minimum values}
To show that the lower bounds we found in the various cases are actually the minimum values assumed by the rank, we have to construct examples of TN rings whose groups of units realise those values. The idea for the construction is to use   the exact sequence \eqref{eq:es-nt} \say{backwards}, finding explicitly a TN
rings such that $A^*$ fits into the sequence, after the other terms are \say{prescribed}. For the reasons we have already explained, these groups cannot be chosen as freely as possible; these limitations are matched in the statement of the next result.

 \begin{prop}\label{proposition: construction}
	Let $G$ be a finitely generated abelian group with $\varepsilon(G)=\varepsilon\ge 1$, and let $k\ge 1$.  Suppose that $G$ is realisable as $B^{\ast}=G$, where $B$ satisfies the following properties:
\begin{itemize}
	\item $B$ is torsion-free and reduced. 
	\item $B$ is a $k$-ring and it is finitely generated over $\Z[\zeta_k]$ by some elements $\gamma_1,\ldots,\gamma_m$.  
	\item The ideal of $B$ generated by $\gamma_1,\ldots,\gamma_m$ intersects trivially $\Z[\zeta_{k}]$.
\end{itemize}
If $H$ is a finite abelian group, whose order is odd and coprime to $k$, such that $H_q$ is a $\lambda(q,k)$-power for all prime divisors $q$ of $|H|$, then the group 
\begin{equation*}
	H\times G
\end{equation*}
is realisable in the class of TN rings. 
\end{prop}

\begin{proof}
Consider the decomposition of $H$ into indecomposable factors
\begin{equation*}
	H\cong \prod_{j=1}^t (\Z/q_j^{b_j}\Z)^{\lambda_j},
\end{equation*}
where $\lambda_j=\lambda(q_j,k)$.

Let $Q_j$ to be a prime ideal of $\Z[\zeta_{k}]$ containing $q_j$. Consider the polynomial ring $B[x_1,\ldots,x_t]$, call $\mathcal{Q}_j$ the ideal of $B[x_1,\ldots,x_t]$ generated by $Q_j$, and set
\begin{equation*}
		J= \sum_{j=1}^t \mathcal{Q}_j^{b_j}x_j+(\gamma_ix_j,x_jx_{j'})_{i=1,\ldots, m,\, j,j'=1,\ldots t}.
	\end{equation*} 
By definition $J$ is an ideal of $B$.
Let 
	\begin{equation*}
		A=\frac{B[x_1,\ldots,x_t]}{J}.
	\end{equation*}
As $J\cap B=\{0\}$, $A$ contains $B$.

By construction, the elements $\overline{x_j}$ (here the bar denotes that class modulo $J$) are nilpotent and torsion, and the isomorphisms
\begin{equation*}
	\frac{A}{(\overline{x_1},\ldots,\overline{x_t})}\cong \frac{B[x_1,\ldots,x_t]/J}{(x_1,\ldots,x_t)/J}\cong \frac{B[x_1,\ldots,x_t]}{(x_1,\ldots,x_t)}\cong  B
\end{equation*}	
implies that 
\begin{equation*}
	\n=\nt=(\overline{x_1},\ldots,\overline{x_t})
\end{equation*}
and that $A$ is a TN ring. 
In addition: 
\begin{itemize}
	\item $\n\cong 1+\n$ via $\overline x\mapsto 1+\overline x$, as $\overline{x_j}^2=0$ for all $j$.
	\item As a group, $\n\cong \oplus_{j=1}^t (\overline{x_j})$, because $\overline{x_j}\overline{x_{j'}}=0$ for all $j,j'$. 
\end{itemize}

In order to compute the group structure of $(\overline{x_j})$, we need to describe the annihilator $\Ann(\overline{x_j})$ in $A$ of $\overline{x_j}$, for all $j$. To lighten the notation, we consider the case $j=1$, as the others are computed in the same way. 
Clearly, 
\begin{equation*}
	\mathcal{Q}_1^{b_1}+({\gamma_i},\overline{x_j})_{i,j}\subseteq \Ann(\overline{x_1}).
\end{equation*}
Conversely, suppose that $\alpha\in \Ann(\overline{x_1})$. By the description of $A$, we can think of $\alpha$ as a polynomial with coefficients in $\Z[\zeta_{k}]$ computed in the elements $\gamma_i$ and $\overline{x_j}$. This implies that
\begin{equation*}
	0=\alpha \overline{x_1}=c_0\overline{x_1},
\end{equation*}
where $c_0\in \Z[\zeta_{k}]$, that is, 
\begin{equation*}
	c_0x_1\in \sum_{j=1}^t \mathcal{Q}_j^{b_j}x_j+(\gamma_ix_j,x_jx_{j'})_{i,=1,\ldots, m,\, j,j'=1,\ldots t}.
\end{equation*}
As we are working in a polynomial ring,  we may  evaluate in $x_2=\cdots=x_t=0$ to obtain that 
\begin{equation*}
	c_0x_1\in \mathcal{Q}_1^{b_1}x_1+ (\gamma_ix_1,x_jx_{1})_{i,=1,\ldots, m,\, j,=1,\ldots t},
\end{equation*}
that is,
\begin{equation*}
	c_0\in \Z[\zeta_k]\cap \left( \mathcal{Q}_1^{b_1}+ (\gamma_i,x_j)_{i=1,\ldots, m,\, j=1,\ldots t}\right).
\end{equation*}
The assumption on the ideal generated by the $\gamma_i$ implies the right-hand side is exactly $Q_1^{b_1}$, and the claim follows. The same argument also shows that the composition
\begin{equation*}
	\Z[\zeta_k]\hookrightarrow A\to \frac{A}{\mathcal{Q}_1^{b_1}+ (\gamma_i,\overline{x_j})_{i=1,\ldots, m,\, j=1,\ldots t}}
\end{equation*}
is surjective with kernel $Q_1^{b_1}$, and taking in consideration Proposition~\ref{proposition: cyclotomic quotients} we conclude that,  for all $j$,
\begin{equation*}
	\frac{A}{\Ann(\overline{x_j})}\cong \frac{\Z[\zeta_{k}]}{Q_j^{b_j}}\cong (\Z/q_j^{b_j}\Z)^{\lambda_j}, 
\end{equation*}
and thus
\begin{equation*}
	1+\nt\cong \nt\cong \bigoplus_{j=1}^t (\overline{x_j})\cong \bigoplus_{j=1}^t \frac{A}{\Ann(\overline{x_j})}\cong \prod_{j=1}^t (\Z/q_j^{b_j}\Z)^{\lambda_j}\cong H
\end{equation*}
as groups. 

Finally,  since our assumptions guarantees that short exact sequence
\begin{equation*}
	1\to 1+\nt \to A^{\ast}\to B^{\ast}\to 1
\end{equation*}
splits, 
we get
\begin{equation*}
	A^{\ast}\cong  (1+\nt) \times C^{\ast}\cong H\times G. \qedhere
\end{equation*}

\end{proof}

\begin{remark}\label{remark: increase connect}
	Notice that if $B=\Z[\zeta_{k}][\gamma_1,\ldots, \gamma_m]$ is a ring that satisfies the hypothesis in the previous proposition, then the same holds for 
	\begin{align*}
		C&=B[u_1,\ldots, u_r,u_{1}^{-1},\ldots, u_{r}^{-1}]\\
		&=\Z[\zeta_{k}][\gamma_1,\ldots, \gamma_m,u_1-1,\ldots, u_r-1,u_{1}^{-1}-1,\ldots, u_{r}^{-1}-1].
	\end{align*}
	In the case in which $B$ is also connected, then, as stated in Proposition~\ref{prop: laurent}, 
	\begin{equation*}
		C^{*}\cong B^*\times \Z^r.
	\end{equation*}
	This means that if $H\times G$ is realisable as in the previous proposition and $B$ is connected, then 
	\begin{equation*}
		H\times G\times \Z^r
	\end{equation*}
	is also realisable in the class of TN rings.
\end{remark}

\begin{remark}
The rings $B$ constructed in Examples~\ref{example: torsion-free not 1} and~\ref{example: torsion-free 1} satisfy the hypothesis of the Proposition~\ref{proposition: construction} (and are also connected by Remark~\ref{remark: these are connected}), when we choose  the generators $\gamma_i=\beta_i-1$, where the $\beta_i$'s are the elements defined in those examples.
\end{remark}

We see now some examples. 

\begin{example}\label{ex:s0}
	Let $T\in\geps$, and assume $s_0\ne 1$. We apply the construction of Proposition~\ref{proposition: construction}, with $k=2^\varepsilon$ and 
\begin{equation*}
	G=\Z/2^{\varepsilon}\Z\times \prod_{i=1}^{s_0} T_{p_i}\times \Z^{g\left(\Z/2^{\varepsilon}\Z \times  \prod_{i=1}^{s_0} T_{p_i}\right)}.
\end{equation*}
	The idea is that we can refer to Example~\ref{example: torsion-free not 1} to construct a torsion-free and reduced ring $B$ with $B^{*}\cong G$. 
	
	Now we observe that the group 
\begin{equation*}
	H=\prod_{j=1}^{t_0} (V_{q_j})^{\lambda_{q_j}}
\end{equation*} 
satisfies the hypothesis of Proposition~\ref{proposition: construction}, and therefore we can construct explicitly a TN ring $A$ with 
\begin{equation*}
	A^*\cong H\times G\cong T\times \Z^{g\left(\Z/2^{\varepsilon}\Z \times  \prod_{i=1}^{s_0} T_{p_i}\right)}.
\end{equation*}
To conclude, as $B$ is connected, we can reason as in Remark~\ref{remark: increase connect} to find that we can also construct TN rings $A_l$ with
\begin{equation*}
	A_l^*\cong T\times \Z^{g\left(\Z/2^{\varepsilon}\Z \times  \prod_{i=1}^{s_0} T_{p_i}\right)}\times \Z^l
\end{equation*}
for all $l\ge 1$.

\end{example}

\begin{example}\label{ex:s0=1-1}
	Let $T\in\geps$ with $s_0=1$. In this case,
	\begin{equation*}
	T\cong \Z/2^{\varepsilon}\Z\times T_p \times \prod_{j=1}^{t_0} (V_{q_j})^{\lambda_{q_j}}
\end{equation*}
with $T_p\cong \prod_{i=1}^v \Z/p^{a_i}\Z$.
 Suppose that for all $j$, $T_{q_j}$ is a $\lambda(q_j,2^{\varepsilon}p^{a_1})$-power.  We apply the construction of Proposition~\ref{proposition: construction}, with $k=2^\varepsilon p^{a_1}$ and 
\begin{equation*}
	G=\Z/2^{\varepsilon}\Z\times T_p\times \Z^{g(\Z/2^{\varepsilon}\Z \times T_p)}.
\end{equation*}
	The idea is that we can refer to Example~\ref{example: torsion-free 1} to construct a torsion-free and reduced ring $B$ with $B^{*}\cong G$. We observe now that the group 
\begin{equation*}
	H=\prod_{j=1}^{t_0} (V_{q_j})^{\lambda_{q_j}}
\end{equation*} 
satisfies the hypothesis of Proposition~\ref{proposition: construction}, and therefore we can construct explicitly a TN ring $A$ with 
\begin{equation*}
	A^*\cong H\times G\cong  T\times \Z^{g(\Z/2^{\varepsilon}\Z \times T_p)}.
\end{equation*}
To conclude, as $B$ is connected, we can reason as in Remark~\ref{remark: increase connect} to find that we can also construct TN rings $A_l$ with
\begin{equation*}
	A_l^*\cong T\times\Z^{g(\Z/2^{\varepsilon}\Z \times T_p)}\times \Z^l
\end{equation*}
for all $l\ge 1$.

\end{example}

\begin{example}\label{ex:s0=1-2}
	Let $T\in\geps$ a group with $s_0=1$.  As before,
	\begin{equation*}
	T\cong \Z/2^{\varepsilon}\Z\times T_p \times \prod_{j=1}^{t_0} (V_{q_j})^{\lambda_{q_j}}
\end{equation*}
with $T_p\cong \prod_{i=1}^v \Z/p^{a_i}\Z$. But this time we do not assume that for all $j$, $T_{q_j}$ is a $\lambda(q_j,2^{\varepsilon}p^{a_1})$-power. We apply the construction of Proposition~\ref{proposition: construction}, with $k=2^\varepsilon$ and 
\begin{equation*}
	G=\Z/2^{\varepsilon}\Z \times T_p \times \Z^{g\left(\Z/2^{\varepsilon}\Z \times  T_p\right)+g(\Z/2^{\varepsilon}\Z)}
\end{equation*}
	The idea is that we can refer to Example~\ref{example: torsion-free not 1} to construct a torsion-free and reduced ring $B$ with $B^{*}\cong G$.
We observe now that the group 
\begin{equation*}
	H=\prod_{j=1}^{t_0} (V_{q_j})^{\lambda_{q_j}}
\end{equation*} 
satisfies the hypothesis of Proposition~\ref{proposition: construction}, and therefore we can construct explicitly a TN ring $A$ with 
\begin{equation*}
	A^*\cong H\times G\cong T\times \Z^{g\left(\Z/2^{\varepsilon}\Z \times T_p\right)+g(\Z/2^{\varepsilon}\Z)}.
\end{equation*}
In addition, as $B$ is connected, we can reason as in Remark~\ref{remark: increase connect} to find that we can also construct TN rings $A_l$ with
\begin{equation*}
	A_l^*\cong T\times \Z^{g\left(\Z/2^{\varepsilon}\Z \times  T_p\right)+g(\Z/2^{\varepsilon}\Z)}\times \Z^l
\end{equation*}
for all $l\ge 1$.
\end{example}

\subsection{The main result} 

We can finally summarise the previous discussion into the following result. Recall that if  $T\in\geps$, then the Sylow subgroups of $T$ fulfil the following: $T_2\cong \Z/2^{\varepsilon}\Z$ ($\varepsilon\ge 1$), and  for the primes $p_1,\dots,p_{s_0}$ such that $T_{p_i}$ is not a $\lambda(p_i,2^{\varepsilon})$-power, then 
\begin{equation*}
	\Prank T_{p_i}<\lambda(p_i,2^{\varepsilon}). 
\end{equation*}
When $s_0=1$, we simply write $p_1=p$ and
\begin{equation*}
	T_p\cong \prod_{i=1}^{v}\Z/p^{a_i}\Z,
\end{equation*}
where $1\le a_1\le\cdots\le a_v$.

\begin{theorem}\label{theorem: main realisable}
	Let $T\in\geps$.
	The group $T\times \Z^r$ is realisable in the class of TN rings if and only if 
	\begin{equation*}
		r\ge r(T),
	\end{equation*}
	where
		\begin{equation*}
		r(T)=\begin{cases}
			 g(\Z/2^{\varepsilon}\Z\times \prod_{i=1}^{s_0}T_{p_i} )& \text{\emph{in case }} C1,\\
			 \\
  g\left( \Z/2^{\varepsilon}\Z\times T_p \right) + g(\Z/2^{\varepsilon}\Z)& \text{\emph{in case }} C2.
		\end{cases}
	\end{equation*}
The cases $C1$ and $C2$ are described as follows:

	\begin{itemize}
		\item $C1:$ $s_0\ne 1$, or  $s_0=1$ and for all odd prime divisors $q\ne p$ of $|T|$, $T_q$ is a $\lambda(q,2^{\varepsilon}p^{a_1})$-power
		\item  $C2$: $s_0=1$ and there exists an odd prime divisor $q\ne p$ of $|T|$ such that $T_q$ is not a $\lambda(q,2^{\varepsilon}p^{a_1})$-power.
\end{itemize}
\end{theorem}
\begin{proof}
Propositions~\ref{prop:lb1} and \ref{prop: s0=1 second} show the lower bounds for case $C1$ and $C2$, respectively. 

On the other hand, Examples~\ref{ex:s0} and \ref{ex:s0=1-1} show that the value $r(T)$ is reached in case $C1$, and in
 Example~\ref{ex:s0=1-2} is presented a group of case $C2$ realising the value $r(T)$ for that case.
 Finally,  the rings of those examples have a connected torsion free quotient modulo $\nt$, so by Remark~\ref{remark: increase connect} all groups $T\times\Z^r$ with $r\ge r(T)$ are realisable.
  \end{proof}

We detail now an application of the main result in the case of finitely generated  abelian group with cyclic torsion part.
 Let $n=2^{\varepsilon}ab$, where $\varepsilon\ge 1$, $a=\prod_{i=1}^{s_0} p_i^{a_i}$, and $b=\prod_{j=1}^{t_0} q_j^{b_j}$. Here the $p_i$ are distinct primes such that $p_i\not\equiv 1\pmod{2^{\varepsilon}}$, and the $q_j$ are distinct primes such that $q_j\equiv 1\pmod{2^{\varepsilon}}$.  

\begin{theorem}
\label{teo:zn}
	The group 
	\begin{equation*}
		\Z/n\Z \times \Z^r
	\end{equation*}
	is realisable as group of units of a TN ring if and only if $r\ge r(\Z/n\Z)$, where
	\begin{equation*}
		r(\Z/n\Z)=\begin{cases}
			g(\Z/2^{\varepsilon}a\Z)& \text{\emph{in case }} C1,\\
			 \\
			g(\Z/2^{\varepsilon}a\Z)+g(\Z/2^{\varepsilon}\Z) &\text{\emph{in case }} C2.
		\end{cases}
	\end{equation*} 
The cases $C1$ and $C2$ are described as follows:
\begin{itemize}
\item C1: $s_0\ne 1$, or $s_0=1$ and $q_j\equiv 1\pmod {p_1^{a_1}}$ for all $j$,
\item C2: $s_0=1$ and $q_j\not\equiv 1\pmod {p_1^{a_1}}$ for some $j$.
\end{itemize}
\end{theorem}

\begin{proof}
Let $\ell$ be an odd prime dividing $n$. As $T=\Z/n\Z$ is cyclic, the only way that $T_{\ell}\cong\Z/\ell\Z$ is a $\lambda_{\ell}$-power is when $\lambda_{\ell}=1$, that is, when $\ell\equiv 1\pmod{2^{\varepsilon}}$. The result then follows from Theorem~\ref{theorem: main realisable}.
\end{proof}

As a particular case, we get the cyclic groups one can realise in the class of TN rings.
\begin{corollary}
	The finite cyclic group realisable in the class of TN rings are exactly those of order $2h$, where $h$ is odd, or $4h'$, where $h'$ is odd and divisible only by primes $q$ with $q\equiv 1\pmod 4$. 
\end{corollary}

\begin{proof}
 Let $n=2^{\varepsilon}ab$, where $\varepsilon\ge 1$, $a=\prod_{i=1}^{s_0} p_i^{a_i}$, and $b=\prod_{j=1}^{t_0} q_j^{b_j}$. Here the $p_i$ are distinct primes such that $p_i\not\equiv 1\pmod{2^{\varepsilon}}$, and the $q_j$ are distinct primes such that $q_j\equiv 1\pmod{2^{\varepsilon}}$. We need to show for which $n$
 \begin{equation*}
 	r(\Z/n\Z)=0.
 \end{equation*}
 First, we appeal to Proposition~\ref{prop:k<2} to derive that $\varepsilon=1,2$. In this case, $r(\Z/n\Z)=0$ is equivalent to $a=0$. 
 
 When $\varepsilon=1$, then $a=0$, as every odd prime $q$ clearly satisfies $q\equiv 1\pmod{2}$. We find in this way the first assertion. 
 
 If instead $\varepsilon=2$, then $a=0$ if and only if every odd prime divisor $q$ of $n$ satisfies $q\equiv 1\pmod{4}$, and the assertion follows once again. 
\end{proof}

Recall that by the results of \cite{PearsonSchneider70} (see also~\cite[Corollary 4.2]{dcdAMPA}), the finite cyclic groups realisable in the class of finite rings are those  whose order is the product of a set of pairwise coprime integers of the form 
\begin{itemize}
	\item $p^{\lambda}-1$, for $p$ prime and $\lambda\ge 1$;
	\item $(p-1)p^k$, for $p$ odd prime and $k\ge 1$. 
\end{itemize}

By Proposition~\ref{prop:ps}, combining this result with  Theorem~\ref{teo:zn}, we find the classification of the finitely generated abelian groups with cyclic torsion part realisable as group of units of a ring.

\begin{corollary}\label{corollary: all the cyclic}
Let $m\ge 1$ and $r\ge 0$. The group $\Z/m\Z\times\Z^r$ is realisable as group of units of a ring if and only if 
\begin{enumerate}
	\item $m$ is the product of a set of pairwise coprime integers of the form 
\begin{itemize}
	\item $p^{\lambda}-1$, for $p$ prime and $\lambda\ge 1$,
	\item $(p-1)p^k$, for $p$ odd prime and $k\ge 1$,
\end{itemize}
and $r\ge 0$
	\item  $m\equiv 0\pmod 2$  and 
	\begin{equation*}
		r\ge \min_{d\in X}r(\Z/(m/d)\Z),
	\end{equation*}
	where  $X$ is the set of the divisors $d$ of $m$ which are products of pairwise coprime integers of the form $2^{\lambda}-1$ and  $\gcd(d,m/d)=1$. 
\end{enumerate}
\end{corollary}

\begin{example}
The group $\Z/328\Z$ is not realisable as group of units, as $328=8\cdot 41$ and the numbers are not in the list above. The results of~\cite{JLMS}, ensures that 
\begin{equation*}
	\Z/328\Z\times \Z^r
\end{equation*}
is realisable in the class of torsion-free rings if and only if $r\ge g(\Z/328\Z)=79$.

Now, $41\equiv 1\pmod 8$, so that, with the notation of Theorem~\ref{teo:zn}, we have,
\begin{equation*}
	n=328=8 \cdot 41=2^{\varepsilon}b
\end{equation*}
with $\varepsilon=3$ and $b=41$. We conclude that 
\begin{equation*}
	r(\Z/328\Z)=g(\Z/8\Z)=1,
\end{equation*}
proving that in the class of TN rings 
\begin{equation*}
	\Z/328\Z\times \Z^r
\end{equation*}
is realisable if and only if  $r\ge 1$. 
Finally, since the cyclic group $\Z/328\Z$ is not realisable, $1$ is the minimum rank also when considering all rings with identity.
\end{example}

The assumption for a group $T\in \mathcal{G}(\varepsilon)$ to have $\Prank T_p <\lambda_p$, for the \say{bad} primes $p$, is due to the fact that it forces all $T_p$ to be a subgroup of $B^*_{\tors}$, in case $A$ is a TN ring with $A^*\cong T$ and $B=A/\nt$. If we just assume that $T_p$ is $\lambda_p$-small, then the same conclusion cannot be derived, but we can still find some information in relation to the realisability in the class of TN rings of finite abelian group of a certain form. 

\begin{prop}
Let $T$ be a finite abelian group such that $T_2\cong \Z/4\Z$ and for all prime divisors $p$ of $|T|$ such that $p\equiv 3\pmod{4}$, $T_p$ is $2$-small. Then the following are equivalent:
\begin{itemize}
	\item $T$ is realisable in the class of TN rings.
	\item For all prime divisors $p$ of $|T|$ such that $p\equiv 3\pmod{4}$, $T_p$ is the square of a $1$-small abelian group.
\end{itemize}

\end{prop}

\begin{proof}
All the groups of this form are realisable by Proposition~\cite[Proposition 5.8]{dcdBLMS}, where the rings considered can be seen to be TN, or by an application of the construction of Proposition~\ref{proposition: construction}.

	Conversely, suppose that there exists $p\equiv 3\pmod{4}$ such that  $T_p$ is not the square of an abelian $p$-group. Let $A$ be a TN ring with
\begin{equation*}
	A^{*}\cong T\times \Z^r.
\end{equation*}
We check that $r>0$. Recall that for  $B=A/\nt$, we have $B^*_2\cong   \Z/4\Z$ (see Corollary~\ref{cor:B2allaepsilon}). Because of Proposition~\ref{prop:ABC}, we can assume $A$ to be of the form $A=\Z[(A^*)_{\tors}]$, so $A$ is finitely generated and integral over $\Z$. As, by assumption, $T_p$ is $\lambda_p=\lambda(p,4)$-small but not a $\lambda_p$-power, then $1+\ntp$ cannot be the whole group $T_p$ by Proposition ~\ref{proposition: tn small}; thus  the short exact sequence
\begin{equation*}
	1\to 1+\ntp \to A^{*}_{\tors,p}\to B^{*}_{\tors,p}\to 1
\end{equation*}
implies that $B^{*}_{\tors,p}$ is nontrivial. We obtain that
\begin{equation*}
	r=\rank B^*\ge g(B^{*}_{\tors})\ge g(\Z/2^{\varepsilon}p\Z)> 0.\qedhere
\end{equation*}
\end{proof}

 \section{Realisability of abelian $2$-groups}
\label{sec:2-groups}

In the previous section we have restricted our study to groups with a cyclic Sylow $2$-subgroup. 
The principal reason is that Theorem~\ref{theorem: main}, which is the tool we have employed to obtain control on the torsion part of the nilradical of TN rings, cannot be used for $p=2$. 

Nevertheless,  some results of the realisability of abelian 2-groups are present in the literature. First of all we recall that  if a finite abelian group $T$ is realisable in the class of TN rings, then $\varepsilon(T)\le2$; see Proposition~	\ref{prop:k<2}.
In the class of TN rings, all finite groups with $\varepsilon=1$ are realisable (see \cite[Theorem~5.1]{dcdBLMS}), whereas for $\varepsilon=2$ we know that
 all $2$-groups of the form type   $\Z/4\Z\times P$ are realisable, where $P$ is trivial or 
 \begin{equation*}
 	P\cong \Z/2^{e_1}\Z\times\dots\times\Z/2^{e_{2r}}\Z 
 \end{equation*} 
 with $2\le e_1\le\dots\le e_{2r}$ and $e_{2j}-e_{2j-1}\le1$ for all $j=1,\dots,r$  (see \cite[Theorem 5.8]{dcdBLMS}). Example~\ref{ex:2revisited} below shows that there are realisable groups outside this family.
 
In this section we study the realisability of  groups of the form $\Z/4\Z\times \Z/2^{u}\Z$ where $u\ge 0$, which is the simplest family of 2-groups for which the question of the realisability is open. In Theorems~\ref{teo:realisableTN} and \ref{teo:2not-rel}  we characterise the realisable ones, both in the class of TN rings and in that of general rings. These results allow us to exhibit  an infinite family of non-realisable 2-groups.
\smallskip

For $u\le 2$ the question of realisability is easily solved observing that $\Z^*\cong \Z/2\Z$ and $\Z[i]^*\cong \Z/4\Z$.

In~\cite[Example 2]{dcdBLMS}, it is claimed that also $\Z/4\Z\times \Z/8\Z$ is realisable in the class of TN rings, but, as shown in Example~\ref{ex:2revisited} below, that  example has a mistake. Nevertheless,  in Example~\ref{example: strange behaviour} we realise $\Z/4\Z\times \Z/8\Z$ with a variation of the previous example.

\begin{example}
\label{ex:2revisited}
		Consider 
	\begin{equation*}
		A=\Z[i][x,y]/(x^2-y-1,(1+i)y,y^3).
	\end{equation*}
	With an argument similar to that of the proof of Proposition~\ref{proposition: construction}, we can derive the following facts:
	\begin{itemize}
		\item $A$ is a TN ring. 
		\item $\Ann(\overline y)=((1+i),\overline y^2)$.
		\item $\n=\nt=(\overline y)\cong (\Z/2\Z)^4$.
	\end{itemize}
Note that while $\overline y\in A$ has additive order $2$, we have that $1+\overline y$ has multiplicative order $4$. The same holds for $1+\overline{xy}$, and moreover $(1+\overline y)^2=1+\overline y^2=(1+\overline{xy})^2$. In particular, $1+\overline y$ and $1+\overline{xy}$ generate a subgroup of $1+\n$ of order $8$, isomorphic to $\Z/2\Z\times \Z/4\Z$, not containing $1+\overline{xy^2}$. We conclude that 	\begin{equation*}
		1+\n\cong \Z/2\Z\times \Z/2\Z\times \Z/4\Z.
	\end{equation*}
Now take   $A/\n\cong \Z[i][x]/(x^2-1)$, and note that  the Chinese Remainder Theorem gives an embedding
	\begin{equation*}
	\Z[i][x]/(x^2-1)\hookrightarrow \Z[i][x]/(x-1)\times  \Z[i][x]/(x+1),
\end{equation*}
from which one can easily compute 
\begin{equation*}
		(A/\n)^{\ast}=\langle \overline{\overline x},i\rangle\cong \Z/2\Z\times \Z/4\Z.
	\end{equation*}
	(Here we are using the notation $\overline{\overline x}$ to denote the class of $\overline x$ modulo $\n$). 
	In particular, $A^{\ast}$ has order $2^7$. It surely contains $\langle \overline x,i\rangle$, which has order $2^5$. One can check that
	\begin{equation*}
		\langle\overline x,i,1+\overline{xy}\rangle
	\end{equation*}
	is a subgroup of $A^\ast$ isomorphic to $\Z/2\Z \times \Z/4\Z \times \Z/8\Z$ which does not contain $1+\overline{xy^2}$, hence
	\begin{equation*}
		A^\ast\cong \Z/2\Z\times \Z/2\Z \times \Z/4\Z \times \Z/8\Z.
	\end{equation*}
	Note that in this example, $\n$ and $1+\n$ are not isomorphic, and the short exact sequence \eqref{eq:success} does not split.

\end{example}

\begin{example}\label{example: strange behaviour}Let
	\begin{equation*}
	A=\frac{\Z[i][x,y]}{(y^2,(1+i)y, x^v-y-1,y(x-1))},
\end{equation*} 
where $v\in \{2,4\}$. 
With an argument similar to that of the proof of Proposition~\ref{proposition: construction}, we can derive the following facts:
\begin{itemize}
		\item $A$ is a TN ring. 
		\item $\Ann(\overline y)=((1+i),\overline y, \overline{x}-1)$.
		\item $\n=\nt=(\overline y)\cong \Z/2\Z\cong 1+\n_{\tors}$.
	\end{itemize}
As 
\begin{equation*}
	A/\n\cong \Z[i][x]/(x^v-1), 
\end{equation*}
we  find  $(A/\n)^*\cong\langle  i\rangle \times \langle \overline{\overline x}\rangle \cong \Z/4\Z\times \Z/v\Z$. We get $|A^*|=8v$ and, since $\overline{x}^v=\overline y+1\in \langle \overline x\rangle$, we conclude that
\begin{equation*}
	A^*\cong \Z/4\Z\times \Z/2v\Z. 
\end{equation*}

So,  for $v=4$ we have  realised the group $\Z/4\Z\times\Z/8\Z$. For $v=2$, we obtain again $\Z/4\Z\times \Z/4\Z$, but this examples shows that  $\varepsilon(A^*)$ may be different than $\varepsilon(B^*)$. Note that also this example shows that the short exact sequence \eqref{eq:success} does not necessarily split.
\end{example}

Summarising, $\Z/4\Z\times \Z/2^u\Z$ is realisable in the class of TN rings for $u=0,1,2,3$. 
The following proposition shows that  these are the only cases.
\begin{prop}\label{prop:realisableTN}
The  group  $\Z/4\Z\times \Z/2^{u}\Z$ is realisable in the class of TN rings if and only if $u\le 3$. 
\end{prop}

In order to prove this result, we need consequence of \cite[Theorem 6.1]{Byo07}, reported below in the specialised case of braces.
\begin{theorem}[Byott]\label{teo:realisableTN}
	Let $(N,+,\circ)$ be a brace of order $2^v$, where $v\ge 3$. If $(N,\circ)$ is cyclic, then $(N,+)$ is cyclic. 
\end{theorem}
\begin{corollary}\label{corollary: not cyclic}
	Let $A$ be a $\TN$ ring such that $\varepsilon(A^*)=2$. If $1+\n$ has order $2^v$, where $v\ge 3$, then $1+\n$ is not cyclic.
\end{corollary}

\begin{proof}
Assume, by contradiction, that $1+\n$ is cyclic, so that also $\n$ is cyclic by Theorem~\ref{teo:realisableTN}. By Corollary~\ref{cor:epsilon-group}, $A$ contains a copy of $\Z[i]$ as a subring, and therefore $\n$ is a $\Z[i]$-module. In particular, we deduce that $\n$ is isomorphic to a direct sum of quotients of $\Z[i]$; as $\n$ is cyclic and $|\n|=2^v$, then we derive 
\begin{equation*}
	\n\cong \Z[i]/(1+i)^v
\end{equation*}
by Proposition~\ref{proposition: cyclotomic quotients}. But $\Z[i]/(1+i)^v$  is not a cyclic group, since it is annihilated by $2^{[(v+1)/2]}$ (here $[x]$ denotes the integral part of $x$).
\end{proof}

We can now prove Proposition~\ref{prop:realisableTN}.
\begin{proof}[Proof of Proposition~\ref{prop:realisableTN}]
Wa have just seen that for $u\le3$ the group  is realisable.

Let $u\ge 4$, and suppose that there exists a TN ring $A$ such that
\begin{equation*}
	A^{\ast}\cong \Z/4\Z\times \Z/2^{u}\Z.
\end{equation*}
As $A^*$ is finite, the same holds for $1+\n$. We derive that  $\nt=\n$.  Let $B=A/\n$. As $\varepsilon(A^*)=2$, we can apply Corollary~\ref{cor:epsilon-group} to deduce that $A$ contains an isomorphic copy of $\Z[i]$ as a subring. By Corollary~\ref{cor:B2allaepsilon}, the same holds for $B$; more precisely, the composition
\begin{equation*}
	\Z[i]\hookrightarrow A\to B
\end{equation*}
is a ring injection.
Let $\alpha$ and $\overline\alpha$ be the images of $i$ in $A$ and $B$, respectively. If $\gamma\in A^*$ is an element of order $2^u$, then $A^*$ is the direct product of the subgroups  $\langle\alpha\rangle $ and $\langle\gamma\rangle$; indeed, if their intersection would not be trivial, then $-1\in \langle \gamma \rangle$, so $-1=\gamma^{2^{u-1}}$ and  Proposition~\ref{prop:k<2} forces $u=2$, whereas we are assuming $u\ge 4$. In particular, $A^*=\langle\alpha\rangle\times\langle\gamma\rangle$. Now note that $\overline\alpha$ has still order $4$ in $B^*$, so $B^*\cong A/1+\n$ implies that
\begin{equation*}
	\langle \alpha\rangle\cap 1+\n =\{1\};
\end{equation*}
we deduce that  $1+\n$ is cyclic.   
 
Theorem~\ref{theorem: finite torsion-free} limits the  possibilities for $B^{\ast}$, which has an element $\overline\alpha$ of order 4, to the following:
\begin{equation*}
	B^{\ast}\cong \begin{cases}
	    \Z/4\Z,\\
		\Z/4\Z\times \Z/2\Z,\\
		\Z/4\Z\times \Z/4\Z.
	\end{cases}
\end{equation*}

We can now deal with the various cases.

\begin{itemize}
\item Suppose that $B^{\ast}\cong \Z/4\Z$. Then $B^*=\langle\bar\alpha\rangle$ and $1+\n$ is cyclic of order $2^u$ ($u\ge4$), and this is not possible, because of Corollary~\ref{corollary: not cyclic}.	
 \item If $B^{*}\cong \Z/4\Z\times \Z/2\Z$, then $1+\n$ would be  isomorphic to $\Z/2^{u-1}\Z$, and this gives a contradiction exactly as before. 
 \item If $B^{\ast}\cong \Z/4\Z\times \Z/4\Z$, then $1+\n$ is cyclic of order $2^{u-2}$. Again, Corollary~\ref{corollary: not cyclic} finds a contradiction for  $u\ge 5$.

To conclude, it remains the case $A^{*}\cong \Z/4\Z\times \Z/16\Z$ and $B^{\ast}\cong \Z/4\Z\times \Z/4\Z$. We have $1+\n\cong \Z/4\Z$ generated by $1+y$, say. Arguing as in the proof of Corollary~\ref{corollary: not cyclic}, we get  $\n\cong \Z/2\Z\times \Z/2\Z$, and so $2y=0$. Thus
\begin{equation*}
	1+\n=\{1,1+y,1+y^2,1+y+y^2+y^3\}
\end{equation*}
and $(1+y)^4=1+y^4=1$, meaning that $y^4=0$. As $y$ and $y^2$ are distinct elements of $\n$, 
\begin{equation*}
	\n=\{0,y,y^2,y+y^2\},
\end{equation*}
so we easily obtain that $y^3=0$. (Note that we can also derive $(1+i)y=0$).

Now let $x$ be an element of order $16$ in $A^{*}$. As 
\begin{equation*}
	A^{*}/(1+\n)\cong B^{\ast}\cong \Z/4\Z\times \Z/4\Z,
\end{equation*}
we deduce that $x^4\in 1+\n$ (and clearly $x^4$ has order $4$). So 
\begin{equation*}
	x^4=1+y \text{ or }x^4=1+y+y^2.
\end{equation*}

In addition, $xy$ is a nilpotent element, and thus 
\begin{equation*}
	xy\in\{0,y,y^2,y+y^2\}.
\end{equation*}
The case $xy=0$ implies that the unit $x$ is a zero divisor, a contradiction. Similar conclusion if $xy=y^2$, as in this case $xy^2=0$.

 If $y(x-1)=0$, 
then $y(x^4-1)=y^2=0$ or $y(x^4-1)=y(y+y^2)=y^2=0$, contradiction. Finally, if $xy=y+y^2$, then $x-1-y$ is in the annihilator $\Ann(y)$ of $y$. Recalling that $y^2\in\Ann(y)$, so $x^4\equiv 1+y\pmod{\Ann(y)}$, we conclude that
\begin{equation*}
	 x\equiv 1+y\equiv x^4\equiv (1+y)^4\equiv 1 \pmod{\Ann(y)},
\end{equation*}
and this implies that $y\in \Ann(y)$, a contradiction.\qedhere
\end{itemize}
 \end{proof}

Finally,  we consider the realisability of $\Z/4\Z \times \Z/2^{u}\Z$ without any assumption on the ring.
\begin{theorem}
 \label{teo:2not-rel}
 	Let $u\ge 0$. The group $\Z/4\Z \times \Z/2^{u}\Z$ is realisable if and only if  $u\le3$ or $2^u+1$ is a Fermat's prime.
 \end{theorem}
\begin{proof}
Write $G_u=\Z/4\Z \times \Z/2^{u}\Z$. If $u\le 3$, then $G_u$ is realisable in the class of TN rings as above. If $2^u+1=p$ is a Fermat's prime, then the group of units of the ring $\Z[i]\times\F_p$ is isomorphic to $G_u$.

Conversely, suppose that $A$ is a ring with $A^*\cong G_u$; by Corollary~\ref{cor:sample-rings}, we can assume that $A$ is commutative and finitely generated and integral over its fundamental subring, and therefore apply Proposition~\ref{prop:ps} to write $A=A_1\times A_2$ with $A_1$ finite and $A_2$ TN. If $A_1^*$ is trivial, then $A^*=A_2^*$, so $u\le3$.
If $A_1^*$ is not trivial, then $A_1^*=\Z/4\Z,  \Z/2^{u}\Z,$ or $\Z/4\Z \times \Z/2^{u}\Z$. 

First, $A_1^*=\Z/4\Z$ implies $A_2=\Z/2^{u}\Z$, and this implies $u\le2$ because of Proposition~\ref{prop:k<2}.

Second, suppose that $A_1^*=\Z/2^u\Z$; by the main result of~\cite{PearsonSchneider70}, $\Z/2^u\Z$ is realisable in the class of finite rings if and only if 
\begin{equation*}
	p^{\lambda}-2^u=1
\end{equation*}
for some odd prime $p$ and $\lambda\ge 1$. When $\lambda\ge 2$, the couple $(p,2)$ is a solution of a Catalan equation, and by the Mih\u{a}ilescu's theorem, this implies $p=3$, $\lambda=2$ and $u=3$.  When $\lambda=1$, we find
\begin{equation*}
	p=2^u+1,
\end{equation*} 
 that is, $2^u+1$ needs to be a Fermat's prime. (The only known Fermat's primes $2^u+1$ are obtained with $u=1,2,4,8,16$.)
 In these cases $A=A_1\times \Z[i]$ realise $G_u$.
 
 Finally,  if $A_1^*=G_u$,  then, by Equation~\eqref{formula},  $A_1$ is not local, but it is the product of two finite ring with groups of units $\Z/4\Z$ and $\Z/2^{u}\Z$, respectively. This forces again $p=2^u+1$ to be a Fermat prime, and in this case  $A_1=\F_5\times\F_p$.
 \end{proof}

\bibliographystyle{amsalpha}
\bibliography{biblio}

\end{document}